\documentclass[11pt]{amsart}
\usepackage {enumerate}
\usepackage{setspace}
\usepackage{caption}
\usepackage{hyperref}
\usepackage{esint}
\usepackage{comment}
\usepackage{color}
\usepackage{lipsum}

\newtheorem{theorem}{Theorem}[section]
\newtheorem{lemma}[theorem]{Lemma}

\newtheorem{proposition}[theorem]{Proposition}

\newtheorem{remark}[theorem]{Remark}
\numberwithin{equation}{section}

\DeclareMathOperator{\area}{area}
\DeclareMathOperator{\Hess}{Hess}
\DeclareMathOperator{\vol}{vol}
\DeclareMathOperator{\id}{id}
\DeclareMathOperator{\Lip}{Lip}

\DeclareMathOperator{\dist}{dist}
\DeclareMathOperator{\Div}{div}
\DeclareMathOperator{\tr}{tr}
\DeclareMathOperator{\diam}{diam}

\usepackage{amsfonts}
\usepackage{amssymb}
\usepackage[utf8]{inputenc}

\usepackage[english ]{babel}

\usepackage{bbding}

\usepackage[all]{xy}

\usepackage{tikz-cd}
\usepackage[all]{xy}
\usepackage{amsfonts}
\usepackage{amssymb}
\usepackage{graphicx}
\usepackage{mathtools}
\usepackage{skak}
\usepackage{foekfont}
\usepackage{amsthm}
\usepackage{calligra}
\usepackage[T1]{fontenc}

\usepackage{amsmath,mathdots}

\usepackage{stmaryrd}

\usepackage[utf8]{inputenc}

\usepackage{mathtools}
\usepackage{mathdots}
\usepackage{tikz}

\definecolor{amber(sae/ece)}{rgb}{1.0, 0.49, 0.0}
\newfont{\rsfsten}{rsfs10 scaled 1200}
\usepackage{MnSymbol}
\usepackage{tikz}
\usepackage{amscd}
\usepackage{MnSymbol}
\usepackage{wasysym}

\usepackage{mathrsfs}

\usepackage{pdfcolmk}

\usepackage{graphicx}

\usepackage {enumerate}
\usepackage{setspace}
\usepackage{caption}
\usepackage{hyperref}
\usepackage{esint}
\usepackage{graphicx,amssymb,amsmath,amsthm,mathrsfs}
\usepackage{epstopdf}

\DeclareMathOperator{\Int}{int}

\begin{document}

\title{Area and Gauss-Bonnet inequalities with scalar curvature}
\author{Misha Gromov and Jintian Zhu}
\date{\today}

\begin{abstract} Let $X$ be an $n$-dimensional Riemannian manifold with  "large positive"  scalar curvature. In this paper, we prove in a variety of cases that if $X$ "spreads"
 in $(n-2)$ directions {\it "distance-wise"},
then it {\it can't} much  "spread"
in the remaining 2-directions {\it  "area-wise".}

Here is  a  geometrically transparent example of what we plan  prove in this regard   that  illustrates the idea.

 Let $g$ be a Riemannin metric on $X= S^2\times \mathbb R^{n-2}$, for which the   submanifolds
$$\mbox {$\mathbb R_s^{n-2}=s\times \mathbb R^{n-2}\subset X$ and
$S^2_y= S^2\times y \subset X$}$$
are {\it mutually orthogonal}  at all  intersection points
$$x=(s,y)\in X=\mathbb R_s^{n-2}\cap S^2_y.$$
(An instance of this is
 $g=g(s,y)=\phi(s,y)^2ds^2+\psi(s,y)^2dy^2$.)

Let  the Riemannian metric on    $\mathbb R_s^{n-2}$ induced from $(X,g)$, that is $g|_{\mathbb R_s^{n-2}}$,  be {\it greater than the Euclidean} metric  on $\mathbb R_s^{n-2} =\mathbb R^{n-2}$ for all $s\in S^2$. (This is interpreted as  "large   spread"  of $g$  in the $(n-2)$ Euclidean directions.)

{\sf If the {\it scalar curvature of $g$ is strictly greater  than that of the unit 2-sphere},
$$Sc(g) \geq  Sc(S^2)+\varepsilon=2+\varepsilon, \mbox { }\varepsilon>0,$$
then, provided $n\leq 7$, } (this, most likely, is  unnecessary)

  {\sf there exists a smooth {\it non-contractible}   spherical  surface $S\subset X$, such that
$$area(S)<area(S^2)=4\pi.$$}

 (This says, in a way, that $(X,g)$ "doesn't   spread much   area-wise" in the 2 directions complementary to the Euclidean ones.)
\end{abstract}
\maketitle

 \section {Statement of the Main Results}
Let $X$ be a compact  Riemannian  manifold   of dimension $n\geq 3$ with smooth  boundary, denoted by $\partial X$. In the following, we divide the boundary $\partial X$ into two compact piecewisely smooth $(n-1)$-dimensional manifolds $\partial_{eff}$ and $\partial_{side}$ (possibly empty) such that they have a common boundary in $\partial X$. These boundary portions $\partial_{eff}$ and $\partial_{side}$ will be called {\it the effective boundary} and {\it the side boundary} respectively in the following discussion. \vspace {1mm}

{ \it Example.}    If  $X$ is a 3-dimensional cylinder, that is, the product of the disc $B^2$ by the unit segment,
$X=B^2\times [0,1]\subset \mathbb R^3$, then its ordinary side
$$\mbox{$S^1\times [0,1]\subset \partial (B^2\times [-0,1])=(S^1\times [0,1])\cup (B^2\times \partial [0,1])$}$$
 for $S^1=\partial B^2$  is an instance of what we call  the side boundary.  Relative to this choice of the side boundary, the union of the top and bottom faces $B^2\times\{1\}\cup B^2\times\{0\}$ serves as the effective boundary.\vspace {2mm}

Let  
$$f: (X,\partial_{eff})\to\left([-1,1]^{n-2},\partial [-1,1]^{n-2}\right)$$
be a continuous  map from $X$ to the $(n-2)$-cube such that $f$ sends the effective boundary to the boundary of the cube.
 Let $h\in H_2(X, \partial_{side})$ be the relative homology class of the point pullback of $f$,  that is
  $$\mbox {$h=f^\ast[t]$, for $t\in\Int[-1,1]^{n-2} =[-1,1]^{n-2}\setminus \partial[-1,1]^{n-2}$.}$$
To see the homology class belongs to $H_2(X,\partial_{side})$, we recall the following Poincar\'e duality 
$$
D:H^k(X,\partial_{eff})\to H_{n-k}(X,\partial_{side}).
$$
Now $h$ is homologous to the homology class $D(f^*\omega)$ in $H_2(X,\partial_{side})$, where $\omega$ is the generator of $H^{n-2}\left([-1,1]^{n-2},\partial [-1,1]^{n-2}\right)$. On the other hand, there is also a nice interpretation from differential topology.
For instance, if $f$ is a smooth map, then the pullback of a generic point $t\in \Int[-1,1]^{n-2}$ is a smooth  oriented surface
  $\Sigma_t=f^{-1} (t)\subset X$ with boundary $\partial \Sigma_t\subset \partial_{side}$ and $h$ is equal to the relative homology class of this $\Sigma_t$.
\vspace {1mm}

  Let $\partial_{\mp,i}\subset \partial [-1,1]^{n-2} $, $i=1,...,n-2$, be the pairs of opposite faces of the cube, let
   $\partial_{\mp,i}(X)\subset\partial_{eff}$, denote their $f$-pullbacks,
 \begin{equation}\label{Eq: definition opposite faces}
 \partial_{\mp,i}X=f^{-1}(\partial_{\mp,i})\cap\partial_{eff}\subset \partial_{eff}
 \end{equation}
 and let $d_i$ denote the distances between these "faces" in $X$,
$$d_i=\dist(\partial_{-,i}X,\partial_{+,i}X),\quad i=1,...,n-2.$$

  \vspace{1mm}

 Now  we are ready  to  state our main result, where there is no, a priori, lower bound on the scalar curvature and  that is proven with a use  of {\it stable $\mu$-bubbles}, which we  do only    for  $n\leq 7$  in the present paper.

  \vspace{1mm}

  \textbf {Theorem 1.1.}  {\it
The relative homology class $h$ can be represented by a smooth embedded surface $\Sigma\subset X$,  the   boundary of which is  contained in
$\partial_{ side}$  and   such that  the integrals of the scalar curvature  of $X$ over  all connected components $S$ of $\Sigma$ and of the mean curvature\footnote{Our sign convention is such that the boundaries of {\it convex} domains have {\it positive} mean curvatures.}   of $\partial_{side}$ over $\Theta=\partial S $ satisfy:
\begin{equation}\label{Eq: GB inequality}
\begin{split}
 \int_S  Sc(X,s)ds+ 2\int_{ \Theta} mean.&curv(\partial_{side}, \theta)d \theta \\
 &\leq 4\pi\chi( S)+ C_n(d_i)\cdot area(S),
 \end{split}
\end{equation}
where $\chi( S) $ is the Euler characteristics of $S$ and
$$C_n(d_i)=\frac {4(n-1)\pi^2}{n}\cdot\sum _{i-1}^{n-2} \frac {1}{d^2_i}.$$}
\begin{remark}
In fact, this theorem only makes sense when the scalar curvature and the mean curvature are non-negative. Otherwise, we can always find a compact surface in the given relative homology class such that the integral
$$
\int_S  Sc(X,s)ds
$$
or
$$
\int_{ \Theta} mean.curv(\partial_{side}, \theta)d \theta
$$
becomes arbitrarily negative and so the desired inequality \eqref{Eq: GB inequality} holds trivially. Even though, such surface may be ruled out if we require its interior and boundary to be ``relatively flat'' around the region where the scalar curvature or the mean curvature is negative. There may be a rough analogy of inequality \eqref{Eq: GB inequality} holds if we further require a curvature bound for the portion of surface and its boundary near the negative region.
\end{remark}
\vspace {2mm}

     \hspace {35mm}        {\sc Examples of Corollaries.} \vspace {2mm}

Let $X$ and $X_i$,   $ i=0,1,...,m$, be connected  orientable   Riemannian manifolds, possibly non-compact and with boundaries,  where $X_0$ is  a compact surface   ($dim(X_0)=2$) and let  $f_i: X\to X_i$ be  continuous maps, such that
the map $$f=(f_0,...,f_m): X\to X_0\times ...\times X_m$$
 is  proper (boundary to boundary, infinity-to-infinity) and it
has  {\it non-zero degree}, e.g. $f$ is a homeomorphism.

 Let the boundary of $X$ have {\it non-negative mean curvature}, let  the scalar curvature of $X$ be {\it bounded from below} by $Sc(X)\geq \sigma>0$ and let us finally assume that $n=dim(X)\leq 7$.

\vspace {1mm}

 \textbf {Corollary 1.2.} {\sf  Let $X_i$,   $i=1,...,m$, be  compact without boundaries,  {\it iso-enlargeable} }(see below) {\sf manifolds,  e.g.  they admit metrics with    non-positive sectional curvatures.}

{\it Then $X$ contains a smooth immersed surface $S\subset X$ that is either spherical non-homologous to zero,
  or it is a topological disk with boundary $\partial S\subset \partial X$, which represents a non-zero homology class  in $H_2(X, \partial X)$ and such that
  $$area(S)\leq \frac {4\pi\chi(S)}{\sigma}.$$}

 {\it Definition of iso-enlrageability.} (Compare with [Gro18].) A Riemannian manifold $Y$ is   {\it iso-enlargeable}, if,
for all $d$,
there exist

$\bullet$  compact Riemannian manifolds with boundaries $U_d$ with their dimensions
$dim(U_d)=dim(Y)$,

$\bullet$ locally isometric immersions $e_d:U_d\to Y$,

$\bullet$ proper continuous maps 
$$\phi=\phi_d:(U_d,\partial U_d )\to ([0,1]^k,\partial [0,1]^k), \quad k=dim(U_d),$$ of {\it non-zero degrees}, such that the distances between the pullbacks of the opposite faces of the cube   are bounded from  below by
$$\dist_{U_d} (\phi^{-1}(\partial _{-,i}), \phi^{-1}(\partial _{+,i}))\geq d,\quad  i=1,...,k.$$

  {\it Sketch of the Proof of  Corollary 1.2.}  (See Subsection \ref{Subsec: corollary} for details)  Since the   manifolds $X_i$, $i=1,,...,m$, are iso-enlargeable, their product $\underline X=X_1\times ... \times X_m$ is also iso-enlargeable. Therefore, there exists the above local isometric immersion
  $e_d:U_d\to \underline X$ for all $d>0$.

Denote $\underline f=(f_1,...,f_m):X\to \underline X$ and let
$\tilde e_d:  \tilde U_d\to X$ by the "$f$-pullback" of $e_d$, that is,
$  \tilde U_d\subset X\times \underline X$ is defined as the set of pairs $(x, \underline x)$, such that
$$\underline f(x)=e_d( \underline x)\mbox  { and $\tilde e_d(x, \underline x)=x.$} $$

Theorem 1.1  applied  to these  $\tilde U_d$ delivers surfaces $\tilde S_d\subset \tilde U_d$ with
$area(\tilde S_d)\leq \frac {4\pi\chi(\tilde S_d)}{\sigma}+\varepsilon_d$,  where $\varepsilon_d\to 0$ for $d\to \infty$ and where the required $S\subset X$ is obtained as a (sub)limit of $\tilde e_d(\tilde S_d)\subset X$.

\vspace {1mm}

More accurately, we are able to establish the following

\vspace {1mm}


{
\textbf {Rigidity Theorem 1.3} (Compare  with  [Zhu19].) {\it  Under assumptions of Corollary 1.2, either there exists $S$  representing a non-trivial homotopy class in $\pi_2(X,\partial X)$ with $area(S)< 4\pi\chi(S)/\sigma$ or the universal covering of $X$ is isometric to the Riemannin product, $X=S_\sigma\times \mathbb R^{n-2}$, where $S_\sigma$ is either the 2-sphere or hemisphere
  of constant curvature $\sigma/2$.
}}
 
\vspace {1mm}

The following is another instance of application of Theorem 1.1, where,  unlike Corollary 1.2. some $d_i$ are kept  bounded.\vspace {1mm}

 \textbf {Corollary 1.4.} {\sf Let   $ X_i$, $i=1,...,m,$  be   connected orientable surfaces with inradii\footnote{The {\it inradius} of a  connected Riemannin manifold $X$  is the supremun of the   numbers $R$,  for which  there exists    a closed ball $B_x(R)\subset X$, such that $B_x(R)$ is compact, doesn't intersect the boundary of $X$ and such
that it is {\it not contained} in a smaller ball $B_x(R-\varepsilon$). For instance if $X$ is compact without boundary, then $inrad(X)=diam(X)$ and if $X$ is complete non-compact then $inrad(X)=\infty$.} $inrad(X_i)\geq d_i$, let
  the maps  $f_i: X\to X_i$ be {\it distance decreasing}  and let, as earlier,  the map
 $$f=(f_0,f_1,...,f_m): X\to X_0\times...\times X_m$$
   have {non-zero} degree, where, recall, $dim(X_0)=2$.}

 {\it If the scalar curvature of $X$ is bounded from below by
$$Sc(X)  -\frac {4(n-1)\pi^2}{n}\cdot\sum _{i=1}^{m} \frac {1}{d^2_i}=\sigma_0>0,  $$
and if the boundary of $X$ has non-negative mean curvature, then there exists a smooth surface  $S\subset X$ with boundary contained in $\partial X$,  which represents a non-zero element in $H_2(X, \partial X)$,} ({\sf it can be a sphere, if, e.g. }$\partial X=\emptyset$) {\it  and  such that
$$area (S)\leq \frac {4\pi \chi(S)}{\sigma_0}.$$}

{\it Remark.}  Since  we prove everything in this paper  for $n=dim(X)\leq 7$,   the above applies only
for $m=1$ and 2,  that makes $n=m+2\leq 6$.

Not to lose  $n=7$, we may allow $X=B^2\times S^1\times X_1\times X^ 2$, with the same assumptions on $X_i$ and $f_i$, and with the same conclusion   as in  Corollary 1.2; this  also  makes no difference in the proof (see Proposition 2.1).

\vspace {1mm}

{\it Idea  of the Proof}. Since $inrad(X_i)\geq d_i$,  there exist  {\it locally distance increasing}  immersions $\tilde e_{i, \varepsilon, d}: [0,d_i-\varepsilon] \times [0,d]\to X_i$, for all positive $\varepsilon$ and $d$, that allows an application   of     (a  simple modification of)  the above sketch of the  proof  of Corollary 1.2, now with  $\varepsilon\to 0$ and  $d\to \infty$.\vspace{1mm}

We also investagate the case when the boundary of $X$ has corners.\vspace{1mm}

 \textbf {Corollary 1.5.} {\sf Let $X$ be a Riemannian manifold diffeomorphic to $X_0\times\mathbb R^{n-2} $, where $X_0$ is a planar $j$-gon. Assume that $X$ has nonnegative scalar curvature and that $\partial X$ is mean convex away from the corners. Furthermore, the dihedral angles $\angle_i$, $i=1,2,\ldots,j$ satisfies
 $$
 \angle_i\leq \alpha_i< \pi,\quad \text{and}\quad \sum_{i=1}^j(\pi-\alpha_i)>2\pi.
 $$
Then there is no proper and globally Lipschitz map $\phi:X\to \mathbb R^{n-2}$ such that $X_0$ is homologous to the $\phi$-pullback of a point.
 
 {\it Idea  of the Proof.} Smooth the manifold $X$ such that the equality
 $$
 \sum_{i=1}^j(\pi-\alpha_i)>2\pi
 $$
 becomes
 $$
 \int_{ \Theta=\partial S} mean.curv(\partial_{side}, \theta)d \theta >2\pi
 $$
 for any surface $S$ with free boundary homologous to $X_0$. The result follows from \eqref{Eq: GB inequality} with $d_i=\infty$.}

\section{Proofs of the Main Results}
In this section, we give a detailed proof for our main theorems. Essentially, the proof is based on the {\it $\mu$-bubble method} from \cite{G2019} and the idea of {\it torical symmetrization} technique in \cite{GL1983}, which originated from \cite{FS1980}.

\begin{proof}[Proof for Theorem 1.1]
The consequence is obvious if the relative homology class $h$ is trivial. In fact, we can pick up a small geodesic $2$-sphere $S$ around some point of $X$. Then the inequality holds automatically since the left hand side is almost zero and the right hand side is about $8\pi$.

In the following, we deal with the case where the relative homology class $h$ is non-trivial. Through a slight perturbation for $\partial_{-,1}X$ and $\partial_{+,1}X$, without loss of generality we can assume that they intersect the rest part of boundary $\partial X-(\partial_{-,1}X\cup\partial_{+,1}X)$ in acute angles. Also it is not difficult to construct a smooth function
$$
\phi_1:X \to \left[-\frac{d_1}{2},\frac{d_1}{2}\right]
$$
with
$$
|d\phi_1|\leq 1\quad\text{and}\quad \phi_1^{-1}\left(\pm\frac{d_1}{2}\right)=\partial_{\pm,1}X.
$$
Now we set the following minimizing problem. Let
$$\Omega_0=\{x\in X:\phi_1(x)<0\}$$
and consider the class
\begin{equation*}
\mathcal C=\{\text{Caccippoli sets $\Omega$ in $X$ such that $\Omega\Delta \Omega_0\Subset X-(\partial_{-,1}X\cup\partial_{+,1}X)$}\}.
\end{equation*}
We pick up the function
$$
h_1:\left(-\frac{d_1}{2},\frac{d_1}{2}\right)\to \mathbb R,\quad t\mapsto -\frac{2(n-1)\pi}{nd_1}\tan\left(\frac{\pi t}{d_1}\right),
$$
and define
\begin{equation}\label{Eq: brane functional}
\mathcal B(\Omega)=\mathcal H^{n-1}(\partial^*\Omega\cap \Int X)-\int_{X}(\chi_\Omega-\chi_{\Omega_0})h_1\circ\phi_1 \,\mathrm d\mathcal H^n,\quad \forall\,\Omega\in \mathcal C,
\end{equation}
where $\partial^*\Omega$ represents the reduced boundary of $\Omega$, $\Int X$ denotes the interior part of $X$ and $\chi_\Omega$ is the characteristic function for $\Omega$. From our choice for function $h_1$ and acute angles along the intersection of $\partial_{\pm,1}X$ and their complement, the boundary portions $\partial_{-,1}X$ and $\partial_{+,1}X$ serve as barriers. From geometric measure theory there exists a smooth minimizer $\Omega_1$ for functional $\mathcal B$, whose boundary $Y_1$ is a smooth embedded hypersurface with free boundary that separates $\partial_{-,1}X$ and $\partial_{+,1}X$. Notice that $\phi_1$ is homotopic to a scale of $f_1$ relative to $\partial_{-,1}\cup\partial_{+,1}$, it follows that $Y_1$ represents the relative homology class in $H_{n-1}(X,\partial X-(\partial_{-,1}X\cup\partial_{+,1}X))$ of the point pullback of the map $F_1=f_1$. Since the relative homology class $h$ is non-trivial, the hypersurface $Y_1$ has at least one component intersecting with boundary portions $\partial_{-,2}X$ and $\partial_{+,2}X$ at the same time. Collect all such components of $Y_1$ and still denote the union by $Y_1$ (since the rest components contribute nothing in our following construction). Now the boundary portion $\partial_{\pm,2}Y_1$ defined as the intersection of $\partial Y_1$ and $\partial_{\pm,2}X$ is non-empty and it is clear that $\dist(\partial_{-,2}Y_1,\partial_{+,2}Y_1)\geq d_2$.

Given any smooth variation vector field $V$ on $ Y_1$, the stability yields
\begin{equation}\label{Eq: stability}
\begin{split}
\int_{Y_1}|\nabla_{Y_1}\psi|^2-\frac{1}{2}&\left(Sc(X)-Sc(Y_1)+|\mathring A|^2+\left(\frac{n}{n-1}h_1^2+2h_1'\right)\circ \phi_1\right)\psi^2\mathrm d\sigma\\
&-\int_{\partial Y_1}II_{\partial X}(\nu,\nu)\psi^2\,\mathrm ds\geq 0,
\end{split}
\end{equation}
where $\psi=g(V,\nu)$ is the lapse function corresponding to the unit outer normal vector field $\nu$ of $Y_1$ with respect to $\Omega_1$, the notation $\mathring A$ denotes the trace-free part of the second fundamental form of $Y_1$ in $X$ and $II_{\partial X}$ is the second fundamental form of $\partial X$ in $ X$ with respect to the unit outer normal $\vec n$. Here we omit the calculation and readers can refer to \cite{G2019} for details illustrating how to obtain above inequality.
Following the idea of torical symmetrization, we construct a new warped product manifold as follows. Let $u_1$ be the first eigenfunction of the Jacobi operator (the elliptic operator corresponding to the bilinear form in \eqref{Eq: stability}). In particular, $u_1$ is a positive function satisfying
$$
-\Delta_{Y_1}u_1-\frac{1}{2}\left(Sc(X)-Sc(Y_1)+|\mathring A|^2-\frac{4(n-1)\pi^2}{nd_1^2}
\right)u_1=\lambda_1u_1\geq 0.
$$
and
$$
\frac{\partial u_1}{\partial \vec n}=II_{\partial X}(\nu,\nu)u_1,
$$
where $\lambda_1$ is the corresponding first eigenvalue and we have used the relation
$$
\frac{n}{n-1}h_1^2+2h_1'=-\frac{4(n-1)\pi^2}{nd_1^2}
$$
from our particular choice for $h_1$.
Now we define $X_1=Y_1\times \mathbb S^1$ and $g_1=g_{Y_1}+u_1^2\mathrm ds^2$.
Through a straightforward calculation we see
\begin{equation*}
\begin{split}
Sc(X_1,(y_1,\theta))&=Sc(Y_1,y_1)-2u_1^{-1}\Delta_{Y_1}u_1\\
&\geq Sc(X,y_1)-\frac{4(n-1)\pi^2}{nd_1^2},\quad y_1\in Y_1,\,\,\theta\in \mathbb S^1,
\end{split}
\end{equation*}
and
$$
mean.curv(\partial X_1,(b_1,\theta))=mean.curv(X,b_1),\quad b_1\in B_1=\partial Y_1,\,\,\theta\in \mathbb S^1.
$$

 Let $\partial_{\pm,2}X_1=\partial_{\pm,2}Y_1\times \mathbb S^1$. For boundary portions $\partial_{\pm,2}X_1$, we can repeat our previous argument with slight changes in choices for functions. In fact, we replace function $\phi_1$ by some
$$
\phi_2:Y_1\to \left[-\frac{d_2}{2},\frac{d_2}{2}\right]
$$
with
$$
|d\phi_2|\leq 1\quad\text{and}\quad \phi_2^{-1}\left(\pm\frac{d_2}{2}\right)=\partial_{\pm,2}Y_1,
$$
and replace function $h_1$ by
$$
h_2:\left(-\frac{d_2}{2},\frac{d_2}{2}\right)\to \mathbb R,\quad t\mapsto -\frac{2(n-1)\pi}{nd_2}\tan\left(\frac{\pi t}{d_2}\right).
$$
Of course, since $\phi_2$ is only a function on $Y_1$, the right way is to view $\phi_2$ as an $\mathbb S^1$-invariant function on $X_1$ when we set the minimizing problem on $X_1$. In a similar way we obtain a smooth minimizer $\Omega_2$ for functional $\mathcal B$ in \eqref{Eq: brane functional} after replacements, whose boundary $Y_2$ is a smooth embedded hypersurface with free boundary separating $\partial_{-,2}X_1$ and $\partial_{+,2}X_1$.

One more thing we need to do here is to show the $\mathbb S^1$-invariance of $Y_2$, which then yields $Y_2=\hat Y_2\times \mathbb S^1$ for some hypersurface $\hat Y_2$ in $Y_1$. Due to our definition of class $\mathcal C$, any component $S_2$ of $Y_2$ has zero algebraic intersection number with $\mathbb S^1$, then the $\mathbb S^1$-invariance follows easily from the stability (with a similar reasoning as in \cite{Z2019}). By tracking the relative homology class of $\hat Y_2$ carefully, we see that $\hat Y_2$ is homologous to the point pullback of the map
$$
F_2=(f_1,f_2):X\to [-1,1]^2.
$$
As before, we can construct a new warped product manifold $X_2$ by warping $Y_2$ and $\mathbb S^1$. In particular, we have $X_2=\hat Y_2\times T^2$,
\begin{equation*}
Sc(X_2,(\hat y_2,\Theta))\geq Sc(X,\hat y_2)-\frac{4(n-1)\pi^2}{n}\left(\frac{1}{d_1^2}+\frac{1}{d_2^2}\right),\quad \hat y_2\in \hat Y_2,\,\,\Theta\in T^2,
\end{equation*}
and
$$
mean.curv(\partial X_2,(\hat b_2,\Theta))=mean.curv(X,\hat b_2),\quad \hat b_2\in \hat B_2= \partial \hat Y_2,\,\,\Theta\in T^2.
$$

We continue above procedure until all pairs of $\partial_{\pm,i}$ are handled. Finally, we obtain a smooth embedded surface $\Sigma=\hat Y_{n-2}$ in $X$ representing the relative homology class $h$ associated with the warped product manifold
\begin{equation}\label{Eq: product manifold}
(X_{n-2},g_{n-2})=(\Sigma\times T^{n-2},  g_\Sigma+u_1^2\mathrm ds_1^2+\cdots+u_{n-2}^2\mathrm ds_{n-2}^2)
\end{equation}
satisfying
\begin{equation}\label{Eq: scalar comparison}
Sc(X_{n-2},(s,\Theta))\geq Sc(X,s)-\frac{4(n-1)\pi^2}{n}\sum_{i=1}^{n-2}\frac{1}{d_i^2},\quad s\in \Sigma,\,\Theta\in T^{n-2},
\end{equation}
and
\begin{equation}\label{Eq: mean curvature comparison}
mean.curv({\partial X_{n-2}},(b,\Theta))= mean.curv({\partial X},b),\quad b\in B=\partial \Sigma,\,\Theta\in T^{n-2}.
\end{equation}
From \eqref{Eq: product manifold} one can compute
$$
Sc( X_{n-2},(s,\Theta))=Sc(\Sigma,s)-2\sum_{i=1}^{n-2}u_i^{-1}\Delta_{\Sigma} u_i-2\sum_{1\leq i<j\leq n-2}\langle\nabla_{\Sigma}\log u_i,\nabla_{\Sigma}\log u_j\rangle
$$
and
$$
mean.curv({\partial X_{n-2}},(b,\Theta))=\kappa_{\partial \Sigma}(b)+\sum_{i=1}^{n-2}\nu(\log u_i),
$$
where $\kappa_{\partial \Sigma}$ is the geodesic curvature of boundary curve $\partial \Sigma$ in $\Sigma$ with respect to the unit outer normal $\nu$.
For any connected component $S$ of $\Sigma$, we have
\begin{equation}\label{Eq: gauss bonnet ineq warped extension}
\begin{split}
&\int_SSc( X_{n-2},(s,\Theta))\,\mathrm d\mu +2\int_{\partial S}mean.curv({\partial X_{n-2}},(b,\Theta))\,\mathrm d\sigma\\
&=\int_SSc(S,s)\,\mathrm ds+2\int_{\partial S}\kappa_{\partial S}(b)\,\mathrm db\\
&\qquad-\sum_{i=1}^{n-2} \int_{S}|\nabla_S\log u_i|^2\mathrm d\sigma- \int_{S}\left|\sum_{ i=1}^{n-2}\nabla_S\log u_i\right|^2\mathrm d\sigma\leq 4\pi\chi(S).
\end{split}
\end{equation}
The desired estimate comes from \eqref{Eq: scalar comparison} and \eqref{Eq: mean curvature comparison} substituted into \eqref{Eq: gauss bonnet ineq warped extension}.
\end{proof}

Combining the torical symmetrization technique for $\mu$-bubbles as well as minimal hypersurfaces, a similar argument leads to the following
\begin{proposition}
Let $f:X\to [-1,1]^{n-k-2}\times T^k$ be a continuous map such that it sends $X-\partial_{side}$ to the boundary of $ [-1,1]^{n-k-2}\times T^k$. Then the pullback homology class $h$ can be represented by a smooth surface $\Sigma\subset X$,  the   boundary of which is  contained in
$\partial_{ side}$  and   such that all connected components $S$ of $\Sigma$ satisfy:
$$
 \int_S  Sc(X,s)ds+ 2\int_{ \Theta} mean.curv(\partial_{side}, \theta)d \theta \leq 4\pi\chi( S)+ C_n(d_i)\cdot area(S),
$$
where $\chi( S) $ is the Euler characteristics of $S$ and
$$C_n(d_i)=\frac {4(n-1)\pi^2}{n}\cdot\sum _{i-1}^{n-k-2} \frac {1}{d^2_i}.$$
\end{proposition}

\subsection{Proof of Corollaries}\label{Subsec: corollary}
In the following, we present the detailed proofs for various corollaries mentioned in the introduction.
\begin{proof}[Proof for Corollary 1.2]
First let us show that for any $\epsilon>0$ there is a smooth surface $S\subset X$ representing a non-trivial homology class in $H_2(X,\partial X)$, that is either a sphere or a disk with boundary $\partial S\subset \partial X$ such that
$$
\area(S)\leq \frac{4\pi\chi(S)}{\sigma}+\epsilon.
$$
Since $X_i$ is iso-enlargeable for $i=1,\ldots,m$, so is the product manifold $\underline X=X_1\times \cdots\times X_m$. From the definition, given any $d>0$ there is a compact manifold $U_d$ with non-empty boundary associated with a locally isometric immersion
$
e_d:U_d\to \underline X
$
and a proper continuous map
$
\phi_d:U_d\to [-1,1]^{n-2}
$
with non-zero degree.
Now we pull back the isometric immersion $e_d:U_d\to \underline X$ along the map $\underline f=(f_1,\ldots,f_m):X\to \underline X$. Define
$$
\tilde U_d=\{(x,u)\in X\times U_d:\underline f(x)=e_d(u)\}
$$
and
$$
\tilde e_d:\tilde U_d\to X,\quad (x,u)\mapsto x.
$$
Clearly $\tilde e_d$ is an immersion. So we can pull back the metric $g$ of $X$ on $\tilde U_d$ such that $\tilde e_d$ becomes an isometric immersion. Denote
$$
\tilde f:\tilde U_d\to U_d,\quad (x,u)\mapsto u.
$$
Then we have the following commutative diagram
\begin{equation*}
\xymatrix{
  \tilde U_d \ar[d]_{\tilde e_d} \ar[r]^{\tilde f}
                & U_d \ar[d]^{e_d}  \\
  X \ar[r]^{\underline f}
                & \underline X             .}
\end{equation*}
{ Notice that we also have an associated decomposition for the boundary $\partial \tilde U_d$, that is, we have $\partial\tilde U_d=\partial_{eff}\cup \partial_{side}$, where
$$\partial_{side}=\partial\tilde U_d\cap \tilde e_d^{-1}(\partial X)$$
and $\partial_{eff}$ is the closure of $\partial\tilde U_d-\partial_{side}$ in $\partial\tilde U_d$. Moreover, the composed map 
$$\hat f=\phi_d\circ \tilde f:(\tilde U_d,\partial_{eff})\to \left([-1,1]^{n-2},\partial [-1,1]^{n-2}\right)$$ 
allows us to apply Theorem 1.1.}
Without loss of generality, we assume that $f_i$ are smooth. In this case, map $f_i$ has bounded Lipschitz norm and the same thing holds for $\tilde f$. Let $\partial_{\pm,i}\tilde U_d$ be the boundary portion of $\tilde U_d$ defined in \eqref{Eq: definition opposite faces}. After a rescaling, we deal with the case when $\tilde f$ is distance decreasing. Clearly we have
$$
d_i=\dist_{\tilde U_d}(\partial_{-,i}\tilde U_d,\partial_{+,i}\tilde U_d)\geq d,\quad i=1,2,\ldots,n-2.
$$
Denote $h$ to be the $\hat f$-pullback of a generic point in $[-1,1]^{n-2}$. Then $h$ can be represented by a smooth embedded surface $\Sigma$ possibly with boundary $\partial\Sigma\subset \partial_{side}$ such that any of its components $S$ satisfies
$$
 \int_S  Sc(X,s)ds+ 2\int_{ \partial S} mean.curv(\partial_{side}, \theta)d \theta \leq 4\pi\chi( S)+ C_n(d_i)\cdot area(S).
$$
Combined with the facts
$$
Sc(X)\geq \sigma>0\quad \text{and}\quad mean.curv(\partial X)\geq 0,
$$
it holds
$$
\area(S)\leq 4\pi\chi(S)\left(\sigma-\frac{4(n-1)(n-2)\pi^2}{nd^2}\right)^{-1}=\frac{4\pi\chi(S)}{\sigma}+\epsilon_d
$$
for $d$ large enough, where $\epsilon_d$ is an error term converging to $0$ as $d\to+\infty$.

We need to show that there is at least one component $S_d$ of $\Sigma$ such that the image $\tilde e_d(S_d)$ represents a non-trivial relative homology class in $H_2(X,\partial X)$.
For our purpose, we consider the map
$$
\tilde F=(f_0\circ \tilde e_d,\tilde f):(\tilde U_d,\partial_{side})\to (X_0\times U_d,\partial X_0\times U_d).
$$
From the above commutative diagram, it is easy to verify the fact that $\deg \tilde F=\deg f\neq 0$. Denote $\hat e_d=(\id,e_d):X_0\times U_d\to X_0\times \underline X$. Notice that
$$
(\hat e_d)_*(\tilde F_*(h))=(\deg \tilde F\cdot \deg \phi_d)([X_0])\neq 0\in H_2(X_0\times \underline X,\partial X_0\times \underline X),
$$
we see that the relative homology class $(\tilde e_d)_*(h)$ is non-zero in $H_2(X,\partial X)$. In particular, we can pick up a component $S_d$ of $\Sigma$ whose image $\tilde e_d(S_d)$ represents a non-trivial relative homology class in $H_2(X,\partial X)$. 

{ Now we would like to complete the proof by taking the limit of $S_d$ up to a subsequence. In general, surfaces $S_d$ may not have a uniform bound on its mean curvature, but this can be overcome through a slight modification with the idea from \cite{Zhu2020b}. We will use the flexibility of the choice for functions $h_i$ in the proof of Theorem 1.1. Actually, for any $d>0$ we can construct smooth functions
$$
h_{i,d}:\left(-\frac{d}{2},\frac{d}{2}\right)\to \mathbb R,\quad i=1,2,\ldots,n-2,
$$
satisfying
\begin{itemize}
\item $h_{i,d}'(t)<0$ and
$$
\lim_{t\in\mp d/2}h_{i,d}(t)=\pm\infty.
$$
\item the quantity
$$
\sum_{i=1}^{n-2}\left(2h_{i,d}'(t_i)+\frac{n}{n-1}h_{i,d}^2(t_i)\right),\quad -\frac{d}{2}\leq t_i\leq \frac{d}{2},
$$
is positive outside $[-1,1]^{n-2}$ and no less than $-\bar\epsilon_d$ globally with $\bar\epsilon_d\to 0$ as $d\to +\infty$.
\item $h_{i,d}$ converges to zero smoothly in every compact subset of $\mathbb R$ as $d\to +\infty$.
\end{itemize}
Replace functions $h_i$ by $h_{i,d}$ in the proof of Theorem 1.1, then we can still find desired surfaces $S_d$ with $\area(S_d)\leq 4\pi\chi(S_d)/\sigma+\epsilon_d$ with $\epsilon_d\to 0$ as $d\to+\infty$. Recall that each time we work with a smooth map 
$$\phi_d=(\phi_{1,d},\phi_{2,d},\ldots,\phi_{n-2,d}): (\tilde U_d,\partial_{eff})\to\left(\left[-\frac{d}{2},\frac{d}{2}\right]^{n-2},\partial\left[-\frac{d}{2},\frac{d}{2}\right]^{n-2}\right)$$
such that each component map $\phi_{i,d}$ satisfies $\Lip\phi_{i,d}\leq 1$. Now we are going to analyze the more delicate behavior of surfaces $S_d$.

(i) If some $S_d$ does not intersect the region
$$
K_d=\phi_d^{-1}\left([-1,1]^{n-2}\right),
$$
it follows from the second property of $h_{i,d}$ that we have the modifed estimate
$$
\area(S_d)<\frac{4\pi\chi(S_d)}{\sigma}.
$$
This already delivers the desired surface and we are done.

(ii) Otherwise each $S_d$ intersects with $K_d$. It follows from the torical band estimate in \cite{G2018} (applied to the manifold after symmetrization once again) that surfaces $S_d$ have their diameters bounded by a universal constant $D$ depending only on $\sigma$. Recall that from the symmetrization procedure we obtain the slicing
$$
S_d=\Sigma_{n-2,d}\subset \Sigma_{n-3,d}\subset \cdots\subset\Sigma_{1,d}\subset \Sigma_{0,d}=\tilde U_d.
$$
Take a fixed point $p_d$ and we just need to investigate the local slicing
\begin{equation}\label{Eq: local slicing}
S_d=\Sigma_{n-2,d}\subset B^{\Sigma_{n-3,d}}_{2D}(p_d)\subset\cdots\subset B^{\Sigma_{1,d}}_{2D}(p_d)\subset B^{\Sigma_{0,d}}_{2D}(p_d),
\end{equation}
where $B_{2D}^{\Sigma_{n-3,d}}(p_d)$ is the geodesic $2D$-ball in $\Sigma_{n-3,d}$ center at $p_d$ and the rest are defined in a similar way. Since the Lipschitz norm of $\phi_{i,d}$ is no greater than one, we see that $\Sigma_{i,d}$ lies in the region $\phi_{i,d}^{-1}(-2D-1,2D+1)$ and so the mean curvatures of hypersurface $\Sigma_{i,d}\times T^{i-1}$ in $\Sigma_{i-1,d}\times T^{i-1}$ are bounded by $\delta_d$ with $\delta_d\to 0$ as $d\to +\infty$ (due to the third property of $h_{i,d}$). Similar to the discussion from \cite{SY2017} (but we don't have the singularity issue here), up to a subsequence the local slicing \eqref{Eq: local slicing} after embedded into $X$ will pointed converge to a weighted area-minimizing slicing
$$
S=\Sigma_{n-2}\subset \Sigma_{n-3}\subset \cdots\subset \Sigma_1\subset \Sigma_0.
$$
Notice that the convergence can be locally $C^{2,\alpha}$-graphical from the view of \cite{Simon1976} and the standard elliptic theory since all $\Sigma_{i,d}$ are local minimizers of functional \eqref{Eq: brane functional} and their mean curvatures have uniformly bounded $C^0$ and Lipschitz norm. As a result, $S$ is a smooth compact surface representing a non-trivial homology class in $H_2(X,\partial X)$ which satisfies
$$
\area(S)\leq \frac{4\pi\chi(S)}{\sigma}.
$$
This completes the proof.}
\end{proof}

{Before we give a proof for the Rigidity Theorem 1.3, we introduce the following generalization of Bourguignon-Kazdan-Warner small deformation theorem to manifolds with boundaries and with increase of metrics.
\begin{proposition}\label{Prop: deformation}
Let $(X,g)$ be a compact Riemannian manifold with
\begin{itemize}
\item $Sc(X)\geq \sigma$ for some constant $\sigma$, and
\item $mean.curv.(\partial X)\geq 0$.
\end{itemize} 
Then one of the following happens:
\begin{itemize}
\item[(i)] $(X,g)$ has nonnegative Ricci curvature and convex boundary, or
\item[(ii)] there are a smooth metric $g'\geq g$ and a smooth positive function $u$ on $X$ such that the warped metric $\bar g=g'+u^2\mathrm ds^2$ on $\bar X=X\times \mathbb S^1$ satisfies $Sc(\bar X)>\sigma$ and $mean.curv.(\partial\bar X)\geq 0$.
\end{itemize}
\end{proposition}
\begin{proof}
In this proof, we will show that $(X,g)$ has nonnegative Ricci curvature and convex boundary under the assumption that (ii) is false.

{\it Step 1. $X$ has constant scalar curvature $\sigma$ and vanishing mean curvature.} Otherwise, we take $g'=g$ and show the existence of the desired function $u$. Consider the following functional
\begin{equation}\label{Eq: functional}
Q(v)=\frac{\int_{X}|\nabla v|^2+\frac{1}{2}Sc(X)v^2 d\mu+\int_{\partial X}mean.curv.(\partial X)v^2d\sigma}{\int_X v^2\mathrm d\mu}
\end{equation}
defined on $C^\infty(X)$. Clearly we see $Q(v)\geq \sigma/2$ from the facts $Sc(X)\geq \sigma$ and $mean.curv.(\partial X)\geq 0$. Take the first eigenvalue $\lambda$ and the first eigenfunction $u$ of the functional $Q$ with 
$$\int_Xu^2\mathrm d\mu=1.$$ 
Then we have 
$$
-\Delta_gu+\frac{Sc(X)}{2}u=\lambda u\quad\text{in}\quad X\quad\text{with}\quad\lambda \geq\frac{\sigma}{2},
$$
and
$$
\frac{\partial u}{\partial\nu}+mean.curv.(\partial X)u=0\quad\text{on}\quad \partial X,
$$
where $\nu$ is the outer unit normal of $\partial X$ in $X$.
It is standard that the first eigenfunction $u$ is positive. As a result, if the scalar curvature $Sc(X)$ is strictly greater than $\sigma$ at some point or the mean curvature of $\partial X$ is strictly positive somewhere, then $\lambda$ is strictly greater than $\sigma/2$. From a straightforward calculation, it follows that $\bar X=X\times \mathbb S^1$ equipped with $g+u^2\mathrm ds^2$ satisfies
$$
Sc(\bar X,\bar g)=Sc(X,g)-\frac{2\Delta_g u}{u}=2\lambda>\sigma
$$
and
$$
mean.curv.(\partial \bar X)=mean.curv.(\partial X)+\frac{\partial \log u}{\partial \nu}=0.
$$
This leads to a contradiction since we have assumed that (ii) cannot happen.

{\it Step 2. $X$ has nonnegative Ricci curvature and convex boundary.} This comes from a standard deformation argument but here our manifold has non-empty boundary and we want the metric to increase. As usually, we take $h$ to be a fixed symmetric $2$-tensor on $M$ and investigate the family of metrics 
$$
g_t=g-2th.
$$
As before, we denote $\lambda_t$ to be the first eigenvalue of functional \eqref{Eq: functional} with respect to the metric $g_t$ and $u_t$ to be the corresponding first eigenfunction with
$$
\int_X u_t^2\,\mathrm d\mu_t=\vol(X,g)
$$
and the Robin boundary condition
$$
\frac{\partial u}{\partial\nu}+mean.curv.(\partial X)u=0\quad\text{on}\quad \partial X.
$$
Due to the facts $Sc(X)\equiv\sigma$ and $mean.curv.(\partial X)\equiv 0$, we see $\lambda_0=\sigma$ and $u_0\equiv 1$ at time $t=0$. Recall
\[
\begin{split}
\lambda_t=\vol(X,g)^{-1}\left(\int_X|\nabla_t u_t|^2+\frac{1}{2}\right. &Sc(X,g_t)u_t^2\mathrm d\mu_t\\&+\left.\int_{\partial X}mean.curv.(\partial X,g_t)u_t^2\mathrm d\sigma_t\right).
\end{split}
\]
After taking derivative along $t$ and substituting $u_0\equiv 1$, we obtain
\begin{equation}\label{Eq: derivative eigenvalue}
\begin{split}
\vol(X,g)\lambda_t'(0)=&\frac{1}{2}\int_X\left.\frac{\partial}{\partial t}\right|_{t=0}Sc(X,g_t)\,\mathrm d\mu_g\\
&+\int_{\partial X}\left.\frac{\partial}{\partial t}\right|_{t=0}mean.curv.(\partial X,g_t)\,\mathrm d\sigma.
\end{split}
\end{equation}

In the following, we are going to use
\begin{lemma}\label{Lem: integral scalar}
Let $\{g_t\}_{-\epsilon<t<\epsilon}$ be a smooth family of metric on a compact manifold $X$ with $g'_t(0)=-2h$. Then we have
\begin{equation*}
\begin{split}
\frac{1}{2}\int_X\left.\frac{\partial}{\partial t}\right|_{t=0}Sc(X,g_t)\,\mathrm d&\mu_g+\int_{\partial X}\left.\frac{\partial}{\partial t}\right|_{t=0}mean.curv.(\partial X,g_t)\,\mathrm d\sigma_g\\
&=\int_X\langle h,Ric(X)\rangle\,\mathrm d\mu_g+\int_{\partial X}\langle h,A(\partial X)\rangle\,\mathrm d\sigma_g,
\end{split}
\end{equation*}
where $A(\partial X)$ is the second fundamental form of $\partial X$ in $(X,g_0)$ with respect to the outer unit normal.
\end{lemma}
Here we just leave the proof of this lemma to Appendix \ref{Append: B} and let us continue the previous proof. Clearly, the equation \eqref{Eq: derivative eigenvalue} now becomes
\begin{equation}\label{Eq: integral scalar 2}
\vol(X,g)\lambda_t'(0)=\int_X\langle h,Ric(X)\rangle\,\mathrm d\mu_g+\int_{\partial X}\langle h,A(\partial X)\rangle\,\mathrm d\sigma_g.
\end{equation}
Now let us show the validity of (i) through a contradiction argument and we make the following discussion:

{\it Case 1. The Ricci curvature $Ric(X)$ is negative at some point $p$.} From the continuity we can assume $p$ to be an interior point of $X$ without loss of generality. Since the Ricci tensor can be diagonalized, we can find an orthonormal frame $\{v_i\}_{i=1}^n$ at point $p$ such that $Ric(v_i)=\mu_iv_i$ and $\mu_1<0$. Extend this frame to some neighborhood $U$ of $p$ away from the boundary $\partial X$, still denoted by $v_i$, and denote $\{\omega_i\}_{i=1}^n$ to be the dual frame. Take a nonnegative cut-off function $\eta$ supported in $U$ such that $\eta$ is positive at $p$. Let 
$$
h=-\eta Ric(v_1,v_1)\omega_1\otimes \omega_1.
$$
Notice that the quantity $Ric(v_1,v_1)$ is negative at point $p$. From continuity we can shrink the support of $\eta$ such that the metric $g_t=g-2th$ satisfies $g_t\geq g$ for all $t>0$. Clearly, \eqref{Eq: integral scalar 2} implies
$$
\lambda_t'(0)=\vol(M,g)^{-1}\int_X\eta|Ric(v_1,v_1)|^2\mathrm d\mu>0.
$$
Therefore, for small positive $t$ there is a positive smooth function $u_t$ such that the warped metric $\bar g=g_t+u_t^2\mathrm ds^2$ satisfies
$$
Sc(\bar X,\bar g)=Sc(X,g_t)-\frac{2\Delta_t u_t}{u_t}=2\lambda_t>\sigma.
$$
Again this contradicts to our assumption at the beginning.

{\it Case 2. The second fundamental form $A(\partial X)$ is negative at some point $q$ on $\partial X$.} The argument in this case is similar to Case 1 and it suffices to construct appropriate choice for $h$. Notice that the second fundamental form $A(\partial X)$ is also diagonaliable. So we can pick up an orthonormal frame $\{v_i\}_{i=1}^{n-1}$ on $\partial X$ at point $q$ such that $A(v_i)=\mu_iv_i$ and $\mu_1<0$. Denote $v_n=\nu(q)$ and then $\{v_i\}_{i=1}^{n}$ forms an orthonormal frame of $X$. Extend this frame to a neighborhood $U$ of $q$ and denote $\{\omega_i\}_{i=1}^n$ is the corresponding dual frame. As before, we take $\eta$ to be a nonnegative cut-off function support in $U$ that is positive at $q$. Moreover, we take $\zeta:[0,+\infty)\to \mathbb R$ to be another nonnegative cut-off function such that $\zeta= 1$ around $0$ and $\zeta= 0$ outside $[0,1]$. Take a fixed smooth extension $\tilde A$ of the tensor $A(\partial X)$ to $U$. Define
$$
h=\eta\zeta\left(\frac{\dist(\cdot,\partial X)}{\epsilon}\right)\tilde A(v_1,v_1)\omega_1\otimes\omega_1.
$$
Again from continuity we can shrink the support of $\eta$ such that the metric $g_t=g-2th$ satisfies $g_t\geq g$ for all $t>0$. Now the equation \eqref{Eq: integral scalar 2} becomes
$$
\lambda_t'(0)=\vol(M,g)^{-1}\left(o(1)+\int_{\partial X}\eta|A(v_1,v_1)|^2\mathrm d\sigma\right),\quad\text{as}\quad \epsilon\to 0.
$$
Take $\epsilon$ to be small enough and we obtain $\lambda_t'(0)>0$. The rest argument is the same as before.
\end{proof}
}

{
To prove the Rigidity Theorem 1.3, we also need the following 
\begin{proposition}\label{Prop: Inradius bound}
Let $X$ be a compact Riemannian manifold (possibly with a boundary), whose universal covering splits as $\tilde X=\tilde X_0\times \mathbb R^m$ for some compact simply connected manifold $\tilde X_0$. Let $U$ be a compact Riemannian manifold with non-empty boundary $\partial U$ such that 
\begin{itemize}
\item $\dim U=\dim X$;
\item the boundary $\partial U$ admits a decomposition $\partial U=\partial_{eff}\cup\partial_{side}$, where $\partial_{eff}$ and $\partial_{side}$ are compact submanifolds in $\partial U$ with dimension $\dim U-1$ such that they have a common boundary in $\partial U$;
\item there is a locally isometric map $e:(U,\partial_{side})\to (X,\partial X)$. 
\end{itemize}
Assume that $f:(U,\partial_{eff})\to (B^k(R),\partial B^k(R))$ is a proper smooth 1-Lipschitz map, where $B^k(R)$ is the $R$-ball in the Euclidean $k$-space $\mathbb R^k$, such that the relative homology class from the $f$-pullback of a generic point in $\Int B^k(R)$ is nonzero.
If $k>m$, then $R\leq const(X)$.
\end{proposition}
\begin{proof}
We would like to show that the relative homology class from the $f$-pullback of a generic point has to be zero if $R$ is large enough. For our purpose, let us modify the map $f$ a little bit. Take a Lipschitz
map
$$\Phi:\left( B^k(R),\partial B^k(R)\right)\to\left( \mathbb S^{k}(1),p_0\right)$$
with $\Lip\Phi\leq C_0R^{-1}$ for some universal constant $C_0$ independent of $R$ and define $F=\Phi\circ f$. Then we see $\Lip F\leq C_0R^{-1}$ and that the $F$-pullback of a generic point represents a non-trivial relative homology class.
Furthermore, we can require that $\Phi$ maps the region outside $B^{k}(R/2)$ to the point $p_0$. As a result, $F$ takes the constant value $p_0$ inside the $(R/2)$-neighborhood of $\partial_{eff}$.
When $R$ is large enough, we plan to construct a suitable homotopy from the map $F:(U,\partial_{eff})\to (\mathbb S^{k}(1),p_0)$ to a new map $ F':( U,\partial_{eff})\to (\mathbb S^{k}(1),p_0)$, whose image has zero measure in $\mathbb S^{k}$. Once this has been done, we conclude that the $F$-pullback of a generic point is homologous to zero and this leads to a contradiction.

In practice, we work with the universal covering $\tilde U$ of $U$ and $G$-invariant maps on $\tilde U$, where $G$ is the Deck transformation group of the covering $\pi:\tilde U\to U$. The benefit is that we have nice description for the geometry of $\tilde U$ since the local isometry $e:(U,\partial_{side})\to (X,\partial X)$ can be lifted to a local isometry $\tilde e:(\tilde U,\tilde\partial_{side})\to (\tilde X,\partial\tilde X)$, where $\tilde \partial_{side}\subset \partial\tilde U$ is denoted to be the preimage $\pi^{-1}(\partial_{side})$. Denote 
$$\tilde e=(\tilde e_1,\tilde e_2):\tilde U\to \tilde X =\tilde X_0\times \mathbb R^m.$$ 
It is clear that the first component map $\tilde e_1$ restricted to every component $C$ of preimage $\tilde e_2^{-1}(y)$ for any $y\in \mathbb R^m$ is a locally isometric map to $\tilde X_0$. For every such component $C$, we can denote 
$$\partial_{C,side}=\partial C\cap\partial_{side}$$
and $\partial_{C,eff}$ to be the closure of $\partial C-\partial_{C,side}$ in $\partial C$. Then we have the local isometry $e_C:(C,\partial_{C,side})\to (\tilde X_0,\partial \tilde X_0)$.

From the quantitive topology actually we have the following alternative for these components.
\begin{lemma}\label{Lem: alternative}
Assume that $X_0$ is a compact simply connected Riemannian manifold (possibly with boundary). Let $C$ be a compact Riemannian manifold such that
\begin{itemize}
\item $\dim C=\dim X_0$;
\item the boundary $\partial C$ admits a decomposition $\partial C=\partial_{C,eff}\cup\partial_{C,side}$, where $\partial_{C,eff}$ and $\partial_{C,side}$ are compact submanifolds in $\partial C$ with dimension $\dim C-1$ such that they have a common boundary in $\partial C$;
\item there is a locally isometric map $e_C:(C,\partial_{C,side})\to (X_0,\partial X_0)$. 
\end{itemize}
Then there is a universal constant $Q$ depending only on $X_0$ such that we have the following alternative:
\begin{itemize}
\item either $C$ is isometric to $X_0$;
\item or $\partial_{C,eff}$ is non-empty and all points in $C$ are contained in the $Q$-neihborhood of $\partial_{C,eff}$.
\end{itemize}
\end{lemma}
\begin{proof}
This result follows from \cite[Theorem D]{Rotman2000}. First let us introduce the definition of width for a homotopy and recall this theorem. Let $H_\tau(t)$ be a homotopy connecting two closed curves parametrized by $t\in[0,1]$. The width of homotopy $H_\tau$ is defined to be
$$
W_{H_\tau}=\max_{t\in[0,1]}Length(H_\tau(t)).
$$
With this definition, Rotman proved that if $(X_0,g)$ is a compact simply connected Riemannian manifold satisfying
\begin{itemize}
\item sectional curvature $K\geq -1$ and diameter $d\leq D$;
\item all metric balls with radius less than $c$ are simply connected,
\end{itemize}
then there is a constant $Q_1=Q_1(\dim X_0,D,c)$ such that for any closed curve $\gamma:[0,1]\to X_0$ we can find a homotopy $H_\tau(t)$ of $\gamma(t)$ to a point whose width satisfies
$
W_{H_\tau}\leq Q_1.
$
Clearly the condition on sectional curvature can be guaranteed by rescaling with the sacrifice that the constant $Q_1$ also depends on the lower bound of sectional curvatures and this does not affect our argument below.

Notice that the covering property can be only destroyed by $\partial_{C,eff}$. As a result, if $\partial_{C,eff}$ is empty, then the local isometry $e_C:(C,\partial_{C,side})\to (X_0,\partial X_0)$ is a covering map. It then follows from the simply-connectedness of $X_0$ that $C$ is isometric to $X_0$. Next we deal with the case when $\partial_{C,eff}$ is non-empty.
Let us complete the proof of this lemma by showing a contradiction if there is a point lying outside the $(Q_1+2\diam X_0)$-neighborhood of $\partial_{C,eff}$. Passing to an equivalent metric, we may assume that $X_0$ has convex boundary and correspondingly the boundary portion $\partial_{C,side}$ is convex. Now we can pick up a point $y$ in $C$ such that $\dist(y,\partial_{C,eff})>Q_1+2\diam X_0$. Let $\gamma$ be a unit-speed minimizing geodesic connecting $y$ and the boundary portion $\partial_{C,eff}$. Clearly the geodesic $\gamma$ has length greater than $\diam X_0$ and so its image $e_C(\gamma)$ under the local isometry $e_C$ cannot be any minimizing geodesic in $X_0$. There are two possibilities:

{\it Case 1. There is a conjugate point in $e_C(\gamma)$ within distance $\diam X_0$.} It follows that there is a non-trivial Jacobi field on the geodesic $e_1(\gamma)$. Up to a lifting, the geodesic $\gamma$ also has a non-trivial Jacobi field. However, this is impossible since the geodesic $\gamma$ is minimizing.

{\it Case 2. There is a cut point in $e_C(\gamma)$ within distance $\diam X_0$.} Denote $x=e_C(y)$ and $x_c$ to be the first cut point. By definition there is another minimizing geodesic $\zeta$ connecting points $x$ and $x_c$ with length no greater than $\diam X_0$. Denote $\zeta^{-1}$ to be the inverse path of $\zeta$. The simply-connectness of $X_0$ guarantees that the closed composed curve $\beta:=\zeta^{-1}\ast e_1(\gamma)|_{xx_c}$ can be homotopic to a point with a homotopy, whose width does not exceed $Q_1$. Recall that 
$$
\dist(y,\partial_{C,eff})>Q_1+2\diam X_0,
$$ 
so the lifting property won't be destoryed by $\partial_{C,eff}$ and we can lift the closed curve $\beta$ on $C$ with the cut point $x_c$ lifted to some point $y_c$ on $\gamma$. It turns out that $y_c$ is a cut point on $\gamma$ within distance $\diam X_0$. Again this is impossible since the geodesic $\gamma$ is minimizing.
\end{proof}

Now let $\tilde \partial_{eff}$ be the closure of $\partial\tilde U-\tilde\partial_{side}$ in $\partial \tilde U$. It follows from Lemma \ref{Lem: alternative} that any component of $\tilde e_2^{-1}(y)$ must be isometric to $\tilde X_0$ if it is not contained in $Q$-neighborhood of $\tilde\partial_{eff}$. Let us collect all such $\tilde X_0$-slices with non-empty intersection with the complement $\tilde U_Q$ of $Q$-neighborhood of $\tilde\partial_{eff}$ in $\tilde U$. Since $\tilde U$ is locally isometrically immersed into $\tilde X_0\times \mathbb R^m$, the collection above induces a fiber bundle $\tilde \xi=(\tilde E,\tilde B,\tilde \pi)$, whose total space $\tilde E(\tilde\xi)$ is an open subset of $\tilde U$ containing $\tilde U_Q$ and each fiber is isometric to $\tilde X_0$. Denote $G$ to be the Deck transformation group of the covering $\pi:\tilde U\to U$. Notice that $\tilde\partial_{eff}$ is exactly the preimage $\pi^{-1}(\partial_{eff})$. Therefore, $\tilde\partial_{eff}$ as well as $\tilde U_Q$ is $G$-invariant. Since any closed totally geodesic $(n-m)$-submanifold in $\tilde X_0\times \mathbb R^m$ must be some $\tilde X_0$-slices, the fiber bundle $\tilde\xi$ is $G$-invariant. That is, any element in $G$ maps $\tilde E(\tilde \xi)$ to itself and preserves fibers. 

In the following, we are ready to construct the desired homotopy. Denote 
$$\tilde F= F\circ \pi:(\tilde U,\tilde\partial_{eff}) \to (\mathbb S^{k},p_0)$$
to be the lift of $F$ and we consider its restriction on $\tilde E(\tilde\xi)$. Since the diameter of $X_0$ is bounded, the image of each fiber concentrates in some small geodesic ball of $\mathbb S^{k}$ once $R$ is large enough. In this case, we can conduct an averaging procedure along fibers to obtain a $G$-invariant map $\tilde F':\tilde E(\tilde\xi)\to \mathbb S^k$, which is constant on each fiber. Recall that the map $F$ takes the constant value $p_0$ in $(R/2)$-neighborhood of $\partial_{eff}$ and the same thing holds for $\tilde F$ in $(R/2)$-neighborhood of $\tilde\partial_{eff}$. If $R$ is large enough, the averaged map $\tilde F'$ will take the constant value $p_0$ around $\partial\tilde E(\tilde\xi)$. Through the constant extension we finally obtain a $G$-invariant map defined on the whole $\tilde U$, still denoted by $\tilde F'$. In particular, this induces a map 
$$\tilde F':(\tilde U,\tilde \partial_{eff})\to (\mathbb S^{k},p_0).$$ When $R$ is large enough, the maps $\tilde F$ and $\tilde F'$ are close enough in $C^0$-sense. From the linear homotopy on $\mathbb S^{k}$ along minimizing geodesics, we can obtain a desired homotopy $\tilde\Phi$ between the maps $\tilde F$ and $\tilde F'$. Notice that $\tilde\Phi$ is also $G$-invariant and so it induces a homotopy $\Phi$ between the maps $F$ and $F'$, where $F'$ is given by $\tilde F'\circ\pi^{-1}$.

Now we show that the image of $F'$ has zero measure in $\mathbb S^k$. Equivalently, let us prove this for the image of $\tilde F'$. Since the map $\tilde F'$ is constant on each fiber, there is a smooth map $\tilde F'_B:\tilde B(\tilde\xi)\to \mathbb S^k$ such that $\tilde F'=\tilde F'_B\circ \tilde\pi$ and so the image of $\tilde F'$ is the same as that of $\tilde F'_B$. Notice that the dimension of $\tilde B(\tilde\xi)$ equals to $m$, which is less than $k$. So the image of $\tilde F'_B$ has zero measure in $\mathbb S^k$ and we complete the proof.
\end{proof}
}

Now we prove the desired Rigidity Theorem 1.3.

\begin{proof}[Proof of Rigidity Theorem 1.3]
All we need to show is that if there is no compact surface $S$  representing a non-trivial homotopy class in $\pi_2(X,\partial X)$ with $\area(S)< 4\pi\chi(S)/\sigma$, then $X$ splits as the Riemannian product $S_\sigma\times Y$, where $S_\sigma$ is a sphere or a hemi-sphere with constant curvature $\sigma/2$.

{First we point out that $X$ cannot admit a smooth metric $g'\geq g$ and a positive smooth function $u$ such that $(X\times \mathbb S^1,g'+u^2\mathrm ds^2)$ has mean convex boundary and scalar curvature $Sc(X\times \mathbb S^1)>\sigma$. Otherwise, it follows from the proof of Corollary 1.2 (with once more symmetrization along the extra $\mathbb S^1$-component) that there is a smooth surface $S$ in $(X,g')$ representing a non-trivial homotopy class in $\pi_2(X,\partial X)$ with area strictly less than $ 4\pi\chi(S)/\sigma$. The area of $S$ with respect to the metric $g$ can only be smaller and so we obtain a contradiction. As a consequence, Proposition \ref{Prop: deformation} yields that $(X,g)$ has nonnegative Ricci curvature and convex boundary.
}

{ Now we are going to prove the splitting of the given manifold $X$. Since $X$ has non-negative Ricci curvature and convex boundary, the Cheeger-Gromoll splitting theorem (as well as the with-boundary version in Appendix \ref{Sec: A}) yields that the universal covering $\tilde X$ of $X$ splits into $\tilde X_0\times \mathbb R^m$, where $\tilde X_0$ is a compact simply-connected Riemannian manifold with non-negative Ricci curvature and convex boundary. All we need to prove is $m\geq n-2$. If so, then $\tilde X$ must splits into $\tilde X_0\times \mathbb R^{n-2}$ since we know $\pi_2(X,\partial X)\neq 0$ from the proof of Corollary 1.2. From $Sc(\tilde X)\geq \sigma>0$ we also see that $\tilde X_0$ is a $2$-sphere with area no greater than $8\pi/\sigma$ or a disk with area no greater than $4\pi/\sigma$. In both cases, the equality implies that $\tilde X_0$ is standard. That is, $\tilde X_0$ is the round $2$-sphere or hemi-$2$-sphere with constant sectional curvature $\sigma/2$.

Recall from the proof of Corollary 1.2 that for any $d>0$ there is a compact Riemannian manifold $U_d$ with $\partial U_d=\partial_{eff}\cup \partial_{side}$ such that
\begin{itemize}
\item there is a local isometry $ e_d:( U_d,\partial_{side})\to (X,\partial X)$;
\item there is a smooth map
$$f=(f_1,f_2,\ldots, f_{n-2}):( U_d,\partial_{eff})\to \left(\left[-\frac{d}{2},\frac{d}{2}\right]^{n-2},\partial\left[-\frac{d}{2},\frac{d}{2}\right]^{n-2}\right)$$
such that $\Lip f_i\leq 1$ and that the $f$-pullback of a generic point is homologically non-trivial.
\end{itemize}  
After composed with a fixed Lipschitz map from $([-1,1]^{n-2},\partial[-1,1]^{n-2})$ to $(B^{n-2},\partial B^{n-2})$, the map $f$ induces a new map 
$$f':(U_d,\partial_{eff})\to (B^{n-2}(\hat Cd),\partial B^{n-2}(\hat Cd))$$ 
with $\Lip f'\leq 1$ for some universal positive constant $\hat C$ independent of $d$. Also, the $f'$-pullback of a generic point in $\Int B^{n-2}(R)$ is also homologically non-trivial. Since $d$ can be arbitrarily large, it follows from Proposition \ref{Prop: Inradius bound} that $m\geq n-2$. This completes the proof.}
\end{proof}

\begin{proof}[Proof for Corollary 1.4]
The proof is quite similar to that of Corollary 1.2. We deal with the case when manifolds $X$ and $X_i$ are all compact. First notice that if a surface $X_i$ has inradii $d_i$, then for any $\epsilon>0$ and $d>0$ there is a compact manifold $U_{i,\epsilon,d}$ with non-empty boundary associated with a locally isometric immersion $e_{i,\epsilon,d}:U_{i,\epsilon,d}\to X_i$ and a continuous map
$$
\phi_{i,\epsilon,d}=(\phi_{i,\epsilon,d,1},\phi_{i,\epsilon,d,2}):U_{i,\epsilon,d}\to [0,d_i-\epsilon]\times [0,d]
$$
such that $\phi_{i,\epsilon,d}$ has non-zero degree and each of its component map is distance decreasing. The construction is fairly easy and we make the following illustration as in Figure \ref{Fig: expand} instead of tedious calculation.
\begin{figure}[htbp]
\centering
\includegraphics[width=12cm]{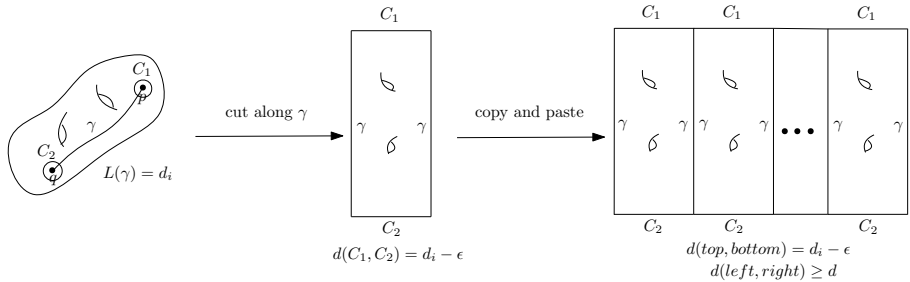}
\caption{The construction of $U_{i,\epsilon,d}$}
\label{Fig: expand}
\end{figure}\\
With these maps $\phi_{i,\epsilon,d}$, the rest of proof is very similar to that of Corollary 1.2 and we omit further details.
\end{proof}

Next we give the proof for Corollary 1.5. Roughly speaking, the following proof follows the line of the sketch mentioned in Section 1, but some technical modifications are needed since the exact meaning of $d_i=+\infty$ is that we will take $d_i$ to be arbitrarily large instead of taking $d_i$ with the value $+\infty$.
\begin{proof}[Proof for Corollary 1.5]
The proof will be divided into two steps.

{\it Step 1.} First let us give a proof with the additional assumption $Sc(X)>0$. We will go through the proof of Theorem 1.1 again, but at this time we will take auxiliary functions $h_i$ such that the quantity
$$
\frac{n}{n-1}h_i^2+2h_i'
$$
is only negative in a bounded interval instead of being a negative constant. This trick helps to guarantee the positivity of
$$
Sc(X)+\left(\frac{n}{n-1}h_i^2+2h_i'\right)\circ\phi_i
$$
even when $Sc(X)$ decays to zero at infinity. 

Now let us give the precise argument. For the sake of contridiction, we assume the existence of a proper and globally Lipschitz map $\phi:X\to \mathbb R^{n-2}$ such that $X_0$ is homologous to the $\phi$-pullback of a point. Without loss of generality, we can further assume that $\phi$ is smooth and the Lipschitz constant is less than one. Since $\phi$ is proper, the subset $X_2=\phi^{-1}([-1,1]^{n-2})$ is compact and so there is a positive constant $\delta$ such that $Sc(X)\geq \delta$ in $X_2$. Take an smooth even function $\eta:\mathbb R\to \mathbb R$ such that $\eta\equiv 0$ outside $[-1,1]$ and $-\delta<\eta<0$ in $[-1,1]$. Let $h_i$, $i=1,\ldots, n-2,$ be the solution of the following ordinary differential equation
$$
\frac{n}{n-1}h_i^2+2h_i'=\frac{\eta}{2(n-2)},\quad h_i(0)=0.
$$
It is easy to show that $h_i$ is a smooth odd function defined on a finite interval $[-d_0,d_0]$ and
$$
\lim_{t\to \mp d_0}h_i(t)=\pm \infty.
$$

In order to avoid the discussion on regularity of $S$ along corners, we have to smooth the manifold $X$ before carry on the argument from the proof of Theorem 1.1. The rounding technique here comes from \cite[P. 699]{G2018}. Given any small positive constant $\epsilon$, we denote $X_\epsilon$ to be the $\epsilon$-neighborhood of the $\epsilon$-core
$$
C_\epsilon=\{x\in X:\dist(x,\partial X)\geq \epsilon\}.
$$
As shown in Figure \ref{smooth}, for $\epsilon$ small $C_\epsilon$ is a Riemannian manifold with $j$ corners close to $X$ and the boundary of $X_\epsilon$ consists of two parts contained in $\partial X$ and the $\epsilon$-level-set of corners respectively.
\begin{figure}[htbp]
\centering
\includegraphics[width=4cm]{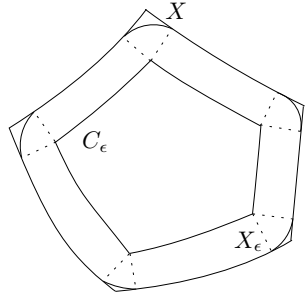}
\caption{The construction of $X_\epsilon$}
\label{smooth}
\end{figure} 

Since $\phi^{-1}([-d_0,d_0]^{n-2})$ is compact, the $\epsilon$-level-set has its mean curvature satisfying the estimate
$$
\frac{1}{\epsilon}+O(1),\quad \text{as}\quad \epsilon\to 0.
$$
On the other hand, given any closed curve in $\partial X_\epsilon$ homologous to $\partial X_\epsilon\cap X_0$, the length of the portion contained in the $\epsilon$-level-set of corners is at least
$$
\sum_{i=1}^j\left(\pi-\alpha_i+o(1)\right)\epsilon\quad\text{as}\quad \epsilon\to 0.
$$
We point out that the boundary of $X_\epsilon$ may not be smooth everywhere but one can start a mean curvature flow from $\partial X_\epsilon$ and pick up a smooth slice that can be arbitrary close to $\partial X_\epsilon$. The important property is still preserved that any closed curve as above has its portion with length no less than $
\sum_{i=1}^j\left(\pi-\alpha_i+o(1)\right)\epsilon
$
contained in boundary portions with mean curvature no less than $\epsilon^{-1}+O(\epsilon)$. We don't plan to bother the audience this technical issue here and just view $X_\epsilon$ as a smooth Riemannian manifold.

Now we can repeat the argument from the proof of Theorem 1.1 on the compact domain $\phi^{-1}([-d_0,d_0]^{n-2})$ and conclude that there exists a smooth embedded surface $S$ homologous to $X_0\cap X_\epsilon$ with free boundary $\partial S\subset \partial X_\epsilon$ such that each component $S'$ satisfies
\begin{equation*}
\begin{split}
\int_{S'}  Sc(X_\epsilon,s)+&\sum_{i=1}^{n-2}\left(\frac{n}{n-1}h_i^2+2h_i'\right)(\phi(s))ds\\
&+ 2\int_{ \Theta'=\partial S'} mean.curv(\partial X_\epsilon, \theta)d \theta \leq 4\pi\chi( S').
\end{split}
\end{equation*}
Since $S$ is homologous to $X_0\cap X_\epsilon$, there is at least one component $S'$ with non-empty boundary. Notice that the quantity
$$
Sc(X_\epsilon,s)+\sum_{i=1}^{n-2}\left(\frac{n}{n-1}h_i^2+2h_i'\right)(\phi(s))
$$
is no less than
$$
\min\left\{\frac{\delta}{2},\min_{\phi^{-1}([-d_0,d_0]^{n-2})}Sc(X_\epsilon,s)\right\}>0,
$$
so the first integral on the left hand side is positive. For the second interal, the mean-convexity of $\partial X_\epsilon$ implies
\begin{equation*}
\begin{split}
\int_{ \Theta'=\partial S'} mean.curv(\partial X_\epsilon, \theta)d \theta&\geq \left(\frac{1}{\epsilon}+O(1)\right)\left(\sum_{i=1}^j\left(\pi-\alpha_i+o(1)\right)\right)\epsilon\\
&=\sum_{i=1}^j(\pi-\alpha_i)+o(1),\quad \text{as}\quad \epsilon\to 0.
\end{split}
\end{equation*}
Let us take $\epsilon$ small enough such that
$$
\sum_{i=1}^j(\pi-\alpha_i)+o(1)>2\pi.
$$
Then we have $4\pi\chi(S')>4\pi$, which leads to a contradiction due to the fact $\chi(S')\leq 1$.

{\it Step 2.} To complete the proof, we show how to deform the scalar curvature to be positive by increasing the dihedral angle $\alpha_i$ a little bit but keeping the mean convexity of the boundary. The discussion will be divided into two cases:

{\it Case 2a.} If the mean curvature of the boundary $\partial X$ is positive somewhere, then we construct a suitable conformal deformation to increase the scalar curvature and we borrow the idea from \cite{K1982}. Let us take the exhausion $\{X_j\}_{j=1}^\infty$ of $X$ such that the boundary point with positive mean curvature is contained in each $X_j$. Consider the functional
$$
Q_j(u)=\frac{\int_{X_j}|\nabla u|^2+\frac{n-2}{4(n-1)}Sc(X)u^2\,\mathrm d\mu+\frac{n-2}{2}\int_{\partial X\cap X_j}mean.curv.(\partial X)u^2d\sigma}{\int_{X_j}u^2\,\mathrm d\mu}.
$$
Clearly there is a positive constant $\mu_j$ such that $Q_j(u)\geq \mu_j$ for any function $u$ in $C^\infty(X_j)$. As a result, we can find a positive smooth function $p$ on $X$ such that for any $u$ in $C^\infty$ it holds
\begin{equation*}
\begin{split}
\int_Xpu^2\,\mathrm d\mu\leq \int_X|\nabla u|^2+&\frac{n-2}{4(n-1)}Sc(X)u^2\,\mathrm d\mu\\
&\quad+\frac{n-2}{2}\int_{\partial X}mean.curv.(\partial X)u^2d\sigma.
\end{split}
\end{equation*}
Using this inequality we are able to construct a smooth positive function $v$ such that
$$
-\Delta v+\frac{n-2}{4(n-1)}Sc(X)v>0\quad \text{in}\quad X
$$
and
$$
\frac{\partial v}{\partial\nu}+\frac{n-2}{2}mean.curv.(\partial X)v\geq 0\quad \text{on}\quad \partial X.
$$
As in \cite[Lemma 2.9]{K1982}, we can further modify $v$ to satisfy $0<\delta\leq v\leq 1$ for some positive constant $\delta$. Define
$$
\bar g=v^{\frac{4}{n-2}}g.
$$
Then $(X,\bar g)$ has positive scalar curvature and mean convex boundary as well as the unchanged dehidral angles. Since the new metric $\bar g$ is equivalent to the original one from our construction, any globally Lipschitz map on $(X,g)$ keeps globally Lipschitz on $(X,\bar g)$. So the previous arguments can be applied to the new manifold $(X,\bar g)$ to deduce a contradiction.

{\it Case 2b.} Now we assume that the boundary $\partial X$ is minimal. In this case, a bending procedure is suggested in \cite[P. 701]{G2018} and here we plan to provide further details based on \cite{LM1984}.

\begin{figure}[htbp]
\centering
\includegraphics[width=8cm]{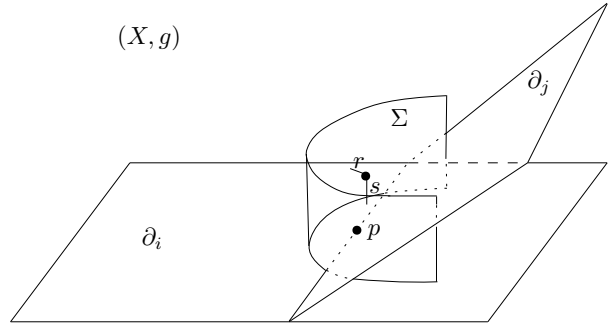}
\caption{The geometry around corners}
\label{bending}
\end{figure} 
As shown in Figure \ref{bending}, we can localize the bending around some point $p$ in the corner that is the intersection of faces $\partial_i$ and $\partial_j$. For convenience, we extend the manifold $X$ such that $\partial_i$ and $\partial_j$ stay inside the extended manifold. Take a hypersurface $\Sigma$ intersecting $\partial_i$ orthogonally such that the point $p$ lies in the tubular neighborhood of $\Sigma$. Denote $r$ and $s$ to be the signed distance function to $\Sigma$ and $\partial_i$ respectively. We make the conventions that points on the right hand side of $\Sigma$ have positive $r$-values and that those points above $\partial_i$ have positive $s$-values. Let 
$$
F(r,s)=s-f(r).
$$
It follows from \cite[P. 403]{LM1984} that the mean curvature of the hypersurface $\{F(r,s)=0\}$ with respect to the unit normal vector field $\nabla F/|\nabla F|$ is 
\begin{equation*}
\begin{split}
H_F=-\frac{f''}{W}&\left(1-\langle\nabla r,\nabla s\rangle^2\right)-\frac{f'}{W}H_r+\frac{1}{W}H_s\\
&\quad +\frac{f'}{W^3}\Hess_r(\nabla s,\nabla s)-\frac{(f')^2}{W^3}\Hess_s(\nabla r,\nabla r),
\end{split}
\end{equation*}
where $H_r$ is the mean curvature of $r$-level-set with respect to $\nabla r$, $H_s$ is the mean curvature of $r$-level-set with respect to $\nabla s$ and
$$
W=\left(1+(f')^2-2f'\langle\nabla r,\nabla s\rangle\right)^{\frac{1}{2}}.
$$
With respect to $\nabla F/|\nabla F|$ we want the mean curvature $H_F$ to be non-positive (and negative somewhere) with some appropriate choice of function $f$. The construction is as follows. Since our bending for $\partial_i$ is only designed around $p$, we can assume $|\langle\nabla r,\nabla s\rangle|\leq 1/2$ in advance. Let $r_0$ be the focal radius of $\Sigma$ and we plan to find a non-negative smooth function $f:[-r_0,r_0]\to \mathbb R $ with small values as well as
$$
0\leq f'\leq \frac{1}{2}\quad \text{and}\quad f''\geq 0.
$$
Notice that we have $|H_r|\leq C$ and $|H_s|\leq Cf$ in $X$ for some universal constant $C$ (the latter needs the fact that dihedral angles are strictly less than $\pi$) and that the Hessians of $r$ and $s$ are bounded as well. We conclude that
$$
H_F\leq -\frac{9}{8}f''+C_1f'+C_2f
$$
for some universal contants $C_1$ and $C_2$. With $r_1$ to be determined later, we define
\begin{equation*}
f(r)=\left\{
\begin{array}{cc}
e^{-\frac{1}{(r-r_1)^2}},&r> r_1;\\
0,&r\leq r_1.
\end{array}\right.
\end{equation*}
Then it follows
$$
H_F\leq e^{-\frac{1}{(r-r_1)^2}}\left(-\frac{9}{2}(r-r_1)^{-6}+\frac{27}{4}(r-r_1)^{-4}+2C_1(r-r_1)^{-3}+C_2\right),\quad r>r_1.
$$
If $r_1$ is sufficiently close to $r_0$, then we obtain the desired bent hypersurface $\{F(r,s)=0\}\cap X$ whose mean curvature with respect to unit outer normal (opposite to $\nabla F/|\nabla F|$) is non-negative (and positive somewhere). Obviously, the bending procedure above does not change dihedral angle too much and so we can still keep the angle condition
$$
 \angle_i\leq \alpha_i< \pi,\quad \text{and}\quad \sum_{i=1}^j(\pi-\alpha_i)>2\pi.
 $$
This reduces to the previous case and we complete the proof.
\end{proof}

\appendix
\section{Splitting theorem for manifolds with nonnegative Ricci and convex boundary}\label{Sec: A}
{ In this appendix, we point out that the original proof of Cheeger-Gromoll splitting theorem can also applied to show
\begin{proposition}
Let $(X^n,g)$ be a complete Riemannian manifold with nonnegative Ricci curvature and convex boundary. Assume there is a geodesic line in $X$. Then $(X,g)$ splits as a Riemannian product $(X_1,g_1)$ and the real line, where $(X_1,g_1)$ also has nonnegative Ricci curvature and convex boundary.
\end{proposition}

We give a sketch for the proof.
\begin{proof}[Sketch of the proof]
Since the boundary is convex, each pair of points can be connected by a length-minimizing geodesic just as in no-boundary case. This guarantees the validity of
$$
\Delta \dist(p,\cdot)\leq \frac{n-1}{\dist(p,\cdot)}
$$
in the distribution sense for any point $p\in X$. Denote $\gamma:(-\infty,+\infty)\to X$ to be the geodesic line. The Busemann functions $B^+$ and $B^-$ are defined to be
$$
B^+(x)=\lim_{t\to+\infty}\left(d_M(x,\gamma(t))-t\right)
$$
and
$$
B^-(x)=\lim_{t\to-\infty}\left(d_M(x,\gamma(t))+t\right).
$$
Denote $B=B^++B^-$. As in the no-boundary case, one can show that the function $B$ is super-harmonic in the distribution sense. Clearly we have
\begin{equation}\label{Eq: 1}
B^++B^-\geq 0\quad \text{in}\quad X
\end{equation}
and
\begin{equation}\label{Eq: 2}
B^++B^-= 0\quad \text{on}\quad \gamma.
\end{equation}

Now we show that $B$ has to be a zero function. If $\gamma$ is contained in the interior of $X$, it follows from the maximum principle that $B$ is identical to zero in $X$. If $\gamma$ is contained in the boundary $\partial X$, then we need to consider the normal derivative of $B$ along $\gamma$. For any point $x_0$, it follows from the length-minimizing property that the distance function $\dist(\cdot,\gamma(t))$ is differentiable at the point $x_0$ when $x_0\neq \gamma(t)$. And it is clear that the normal derivative
$$
\frac{\partial}{\partial \nu}\dist(\cdot,\gamma(t))
$$
vanishes at point $x_0$, where $\nu$ is the outer unit normal of $\partial X$. Take a smooth path $\zeta:[0,\epsilon]\to X$ with $\zeta(0)=x_0$ and $\zeta'(0)=-\nu(x_0)$. Based on the monotonicity of the Busemann functions $B^+$ and $B^-$, it is easy to see
$$
\limsup_{s\to 0^+}\frac{B(\zeta(s))}{s}\leq \lim_{s\to 0^+}\frac{\dist(\zeta(s),\gamma(t))-t+\dist(\zeta(s),\gamma(-t))-t}{s}=0
$$
for any fixed large $t$. On the other hand, from \eqref{Eq: 1} and \eqref{Eq: 2} we have
$$
\liminf_{s\to 0^+}\frac{B(\zeta(s))}{s}\geq 0.
$$
This means that the normal derivative of $B$ along $\gamma$ exists and equals to zero. From strong maximum principle we see that $B$ is again a zero function.

As in the closed case, this implies that $B^+$ and $B^-$ are smooth harmonic functions. From construction we have $|\nabla B^+|\equiv 1$ and it follows from Bochner formula that $\nabla^2 B^+\equiv 0$ and $Ric(X)\equiv 0$. Therefore, the function $B^+$ has no critical point and all level-sets of $B^+$ are totally geodesic. From strong maximum principle we conclude that either some component of $\partial X$ is contained in some level-set of $B^+$ or that $B^+|_{\partial X}$ also has no critical point in $\partial X$. In the first case, $X$ is diffeomorphic to $\partial X\times [0,+\infty)$ and the range of $B^+$ cannot be the entire real line. This is impossible since the image of $B^+|_\gamma$ is already the whole $\mathbf R$. In the second case, $X$ splits topologically as $X_1\times \mathbf R$, where each $X_1$-slice corresponds to a level-set of $B^+$.

Finally we show that $X$ also splits as a Riemannian manifold. It suffices to show that the level-sets of $B^+$ intersects $\partial X$ orthogonally. Let $\gamma'$ be an integral curve of vector field $\nabla B^+|_{\partial X}$ on $\partial X$. Denote $S_t$ to be the level-set $\{B^+=t\}$ diffeomorphic to $X'$. We investigate the function
$$
d(t)=\dist_{S_t}(\gamma\cap S_t,\gamma'\cap S_t).
$$
If $\gamma$ is in the interior of $X$, then $d(t)>0$ for all $t$. If $\gamma$ is contained in $\partial X$, then it is also an integral curve of vector field $\nabla B^+|_{\partial X}$ on $\partial X$. So either $\gamma'$ coincides with $\gamma$ or $d(t)>0$ for all $t$. The proof will be completed if we can show the desired orthogonality from $d(t)>0$ for all $t$. From a direct calculation we see
$$
d'(t)=-\langle\nu,\nabla B^+\rangle
$$
and 
$$
d''(t)=-\left|\nabla B^+|_{\partial X}\right|^{-4}A_{\partial X}(\nabla B^+|_{\partial X},\nabla B^+|_{\partial X})\leq 0.
$$
Since a concave function with a lower bound must be a constant, we see that $\nabla B^+$ is orthogonal to $\nu$ along $\gamma'$. From the arbitrary choice for $\gamma'$, we conclude that the level-sets of $B^+$ intersects $\partial X$ orthogonally.
\end{proof}

\section{Proof of Lemma \ref{Lem: integral scalar}}\label{Append: B}

The proof of Lemma \ref{Lem: integral scalar} will follow from a straightforward calculation (see also \cite{BC2018}) and here we include a detailed calculation for completeness. 

Let $\{g_t\}_{-\epsilon<t<\epsilon}$ be a smooth family of metrics on compact manifolds $X$ with non-empty boundary $\partial X$. Assume $g'_t(0)=-2h$. First we start with the following simple lemma.
\begin{lemma}
For any smooth vector field $U$, $V$ and $W$, we have
\begin{equation}\label{Eq: B1}
\begin{split}
g\left(\left.\frac{\partial}{\partial t}\right|_{t=0}\nabla^{g_t}_UV,W\right)&=-\left(\nabla_Uh\right)(V,W)\\
&-\left(\nabla_Vh\right)(U,W)+\left(\nabla_Wh\right)(U,V).
\end{split}
\end{equation}
\end{lemma}
\begin{proof}
First notice that we have
$$
\left.\frac{\partial}{\partial t}\right|_{t=0}[U,V]=0.
$$
This yields
$$
\left.\frac{\partial}{\partial t}\right|_{t=0}\nabla^{g_t}_UV=\left.\frac{\partial}{\partial t}\right|_{t=0}\nabla^{g_t}_VU.
$$
So the derivative
$$
\left.\frac{\partial}{\partial t}\right|_{t=0}\nabla^{g_t}_UV
$$
is a tensor and we can just make a computation in an orthonormal coordinate $\{x^i\}$ with respect to the metric $g$. It is clear that
\begin{equation*}
\begin{split}
\left.\frac{\partial}{\partial t}\right|_{t=0}\nabla^{g_t}_{\partial_i}\partial_j&=\left(\left.\frac{\partial}{\partial t}\right|_{t=0}\Gamma_{ij,t}^k\right)\partial_k\\
&=-\left(\partial_jh_{ik}+\partial_ih_{jk}-\partial_kh_{ij}\right)\partial_k.
\end{split}
\end{equation*}
This implies
\begin{equation*}
\begin{split}
g\left(\left.\frac{\partial}{\partial t}\right|_{t=0}\nabla^{g_t}_{\partial_i}\partial_j,\partial_k\right)&=-\left(\nabla_{\partial_i}h\right)(\partial_j,\partial_k)\\
&-\left(\nabla_{\partial_j}h\right)(\partial_i,\partial_k)+\left(\nabla_{\partial_k}h\right)(\partial_i,\partial_j).
\end{split}
\end{equation*}
The desired equality \eqref{Eq: B1} now follows from a simple linear combination on both sides.
\end{proof}

We recall the following fact from \cite[Theorem 1.174]{Besse2008}.
\begin{lemma}\label{Lem: B2}
We have
\begin{equation}\label{Eq: B2}
\begin{split}
\left.\frac{\partial}{\partial t}\right|_{t=0}Sc(X,g_t)=&2\langle h,Ric(X,g)\rangle_g\\
&-2\Div_g\left(\Div_gh-d\tr_gh\right).
\end{split}
\end{equation}
\end{lemma}
Next we compute the variation formula of the mean curvature of $\partial X$. For convenience, we take an othornormal coordinate system $\{x^i\}_{i=1}^{n-1}$ on $\partial X$ with respect to the induced metric from $(X,g)$. Denote $s$ to be the distance function to $\partial X$ in $(X,g)$. Then $\{x^i,s\}$ forms a coordinate system of $X$ around $\partial X$. We are going to prove
\begin{lemma}\label{Lem: B3}
\begin{equation}\label{Eq: B3}
\begin{split}
\left.\frac{\partial}{\partial t}\right|_{t=0}&mean.curv.(\partial X,g_t)=2\langle h,A(\partial X,g)\rangle_g\\
&-\,mean.curv.(\partial X,g)\,h(\nu,\nu)
+2(\Div_{\partial X}h)(\nu)-\tr_{\partial X}(\nabla_\nu h).
\end{split}
\end{equation}
\end{lemma}
\begin{proof}
From the definition we know
$$
mean.curv.(\partial X,g_t)=g_t^{ij}A_{t,ij},
$$
where $A_t$ is the second fundamental form of $\partial X$ in $(X,g_t)$ with respect to the outer unit normal $\nu_t$. This implies
$$
\left.\frac{\partial}{\partial t}\right|_{t=0}mean.curv.(\partial X,g_t)=2\langle h,A(\partial X,g)\rangle_g+g^{ij}\left.\frac{\partial}{\partial t}\right|_{t=0}A_{t,ij}.
$$
Now we calculate the second term on the right hand side. From the equation \eqref{Eq: B1} we have
\begin{equation*}
\begin{split}
\left.\frac{\partial}{\partial t}\right|_{t=0}A_{t,ij}&=-\left.\frac{\partial}{\partial t}\right|_{t=0}g_t\left(\nabla^{g_t}_{\partial_i}\partial_j,\nu_t\right)\\
&=2h\left(\nabla_{\partial_i}\partial_j,\nu\right)+\left(\nabla_{\partial_i}h\right)(\partial_j,\nu)\\
&\qquad+\left(\nabla_{\partial_j}h\right)(\partial_i,\nu)-\left(\nabla_{\nu}h\right)(\partial_i,\partial_j)-g\left(\nabla_{\partial_i}\partial_j,\left.\frac{\partial}{\partial t}\right|_{t=0}\nu_t\right),
\end{split}
\end{equation*}
where $\nabla$ is the covariant derivative with respect to metric $g$ and $\nu$ is the outer unit normal of $\partial X$ in $(X,g)$. Since $\{x^i\}$ is an orthonormal coordinate system on $\partial X$, we conclude
$$
g\left(\nabla_{\partial_i}\partial_j,\left.\frac{\partial}{\partial t}\right|_{t=0}\nu_t\right)=g\left(\nabla_{\partial_i}\partial_j,\nu\right)\left\langle\left.\frac{\partial}{\partial t}\right|_{t=0}\nu_t,\nu\right\rangle_g.
$$
Taking the derivative of on both sides of the equation $g_t(\nu_t,\nu_t)\equiv 1$ with respect to $t$, we obtain
$$
\left\langle\left.\frac{\partial}{\partial t}\right|_{t=0}\nu_t,\nu\right\rangle_g=h(\nu,\nu).
$$
As a result, we see
\begin{equation*}
\begin{split}
g^{ij}\left.\frac{\partial}{\partial t}\right|_{t=0}A_{t,ij}=-&\,mean.curv.(\partial X,g)\,h(\nu,\nu)\\
&+2(\Div_{\partial X}h)(\nu)-\tr_{\partial X}(\nabla_\nu h).
\end{split}
\end{equation*}
This complete the proof.
\end{proof}

Now we are ready to prove Lemma \ref{Lem: integral scalar}.
\begin{proof}[Proof for Lemma \ref{Lem: integral scalar}]
From the divergence theorem as well as Lemma \ref{Lem: B2} and Lemma \ref{Lem: B3}, we have
\begin{equation*}
\begin{split}
\frac{1}{2}\int_X&\left.\frac{\partial}{\partial t}\right|_{t=0}Sc(X,g_t)\,\mathrm d\mu_g+\int_{\partial X}\left.\frac{\partial}{\partial t}\right|_{t=0}mean.curv.(\partial X,g_t)\,\mathrm d\sigma\\
&=\int_X\langle h,Ric(X,g)\rangle_g\,\mathrm d\mu_g+2\int_{\partial X}\langle h,A(\partial X,g)\rangle_g\,\mathrm d \sigma_g\\
&\quad+\int_{\partial X}\nu(\tr_g h)-(\Div_gh)(\nu)+2(\Div_{\partial X}h)(\nu)-\tr_{\partial X}(\nabla_\nu h)\,\mathrm d\sigma_g\\
&\qquad\qquad -\int_{\partial X}mean.curv.(\partial X,g)\,h(\nu,\nu)\,\mathrm d\sigma_g.
\end{split}
\end{equation*}
Let us deal with the terms in the third line. Clearly we have
\begin{equation*}
\begin{split}
\nu(\tr_g h)&=\nu\left(g^{ij}h(\partial_i,\partial_j)+h(\partial_s,\partial_s)\right)\\
&=-2\langle A(\partial X,g),h\rangle+\tr_{\partial X}(\nabla_vh)+2\langle h,A(\partial X,g)\rangle+\nu h(\partial_s,\partial_s)\\
&=\tr_{\partial X}(\nabla_vh)+(\nabla_\nu h)(\nu,\nu)
\end{split}
\end{equation*}
and
$$
(\Div_g h)(\nu)=(\Div_{\partial X}h)(\nu)+(\nabla_\nu h)(\nu,\nu).
$$
This yields that the integral in the third line equals to
$$
\int_{\partial X}(\Div_{\partial X} h)(\nu)\,\mathrm d\sigma_g.
$$
We compute
\begin{equation*}
\begin{split}
(\Div_{\partial X} h)(\nu)&=g^{ij}(\nabla_{\partial_i}h)(\partial_j,\nu)\\
&=g^{ij}\left(\nabla_{\partial_i}\left(h(\cdot,\nu)\right)\right)(\partial_j)-g^{ij}h(\partial_j,\nabla_{\partial_i}\nu)\\
&=\Div_{\partial X}(h(\cdot,\nu)|_{\partial X})\\
&\qquad+h(\nu,\nu)\,mean.curv.(\partial X,g)-\langle h,A(\partial X,g)\rangle.
\end{split}
\end{equation*}
Finally we arrive at
\begin{equation*}
\begin{split}
\frac{1}{2}\int_X&\left.\frac{\partial}{\partial t}\right|_{t=0}Sc(X,g_t)\,\mathrm d\mu_g+\int_{\partial X}\left.\frac{\partial}{\partial t}\right|_{t=0}mean.curv.(\partial X,g_t)\,\mathrm d\sigma\\
&=\int_X\langle h,Ric(X,g)\rangle_g\,\mathrm d\mu_g+\int_{\partial X}\langle h,A(\partial X,g)\rangle_g\,\mathrm d \sigma_g.
\end{split}
\end{equation*}
This completes the proof.
\end{proof}
}


\begin{thebibliography}{plain}
\bibitem[Bes08]{Besse2008}
Arthur L. Besse, \emph{Einstein manifolds}, Classics in Mathematics, Springer-Verlag, Berlin, 2008, Reprint of the 1987 edition.

\bibitem[BBN10]{BBN2010}
Hubert Bray, Simon Brendle, and Andr\'e Neves, \emph{Rigidity of
  area-minimizing two-spheres in three-manifolds}, Comm. Anal. Geom.
  \textbf{18} (2010), no.~4, 821--830.
  
\bibitem[BC18]{BC2018}
Ezequiel Barbosa, and Franciele Conrado, \emph{Topological obstructions to nonnegative scalar curvature and mean convex boundary}, 	arXiv:1811.08519.

\bibitem[FCS80]{FS1980}
Fischer-Colbrie, Doris and Schoen Richard, \emph{The structure of complete stable minimal surfaces in $3$-manifolds of nonnegative scalar curvature}. Comm. Pure Appl. Math. 33(1980), no. 2, 199-211.

\bibitem[Gro18]{G2018}
Misha Gromov, \emph{Metric inequalities with scalar curvature}. Geom. Funct. Anal. 28(2018), no. 3, 645-726.

\bibitem[Gro19]{G2019}
Misha Gromov,\emph{ Four lectures on scalar curvature}. arXiv:1908.10612.

\bibitem[Gro20]{G2020}
Misha Gromov, \emph{ No metrics with positive scalar curvatures on aspherical $5$-manifolds}. arXiv:2009.05332.

\bibitem[GL83]{GL1983} M. Gromov and H. B. Lawson, Jr., \emph{Positive scalar curvature and the Dirac operator on complete Riemannian manifolds}. Publ. Math. I.H.E.S., 58(1983), 83--196.

\bibitem[GLZ20]{GLZ20}
Qiang Guang, Martin Man-chun Li, Xin Zhou, \emph{Curvature estimates for stable free boundary minimal hypersurfaces}. J. Reine Angew. Math. 759(2020), 245-264.

\bibitem[Kaz82]{K1982} 
Jerry L. Kazdan, \emph{Deformation to positive scalar curvature on complete manifolds}, Math. Ann. 261 (1982) no.~2, 227--234.

\bibitem[LM84]{LM1984}
Jr., H. Blaine Lawson,  Marie-Louise Michelsohn, \emph{Embedding and surrounding with positive mean curvature}, Invent. Math. 77(1984) no.~3, 399--419.

\bibitem[MN12]{MN2012}
Fernando~C. Marques and Andr\'{e} Neves, \emph{Rigidity of min-max minimal
  spheres in three-manifolds}, Duke Math. J. 161 (2012), no.~14,
  2725--2752.
  
\bibitem[Ric20]{Thomas2020}
Thomas Richard, \emph{ On the $2$-systole of stretched enough positive scalar curvature metrics on $ S^2\times S^2$}. arXiv:2007.02705.

\bibitem[Rot00]{Rotman2000}
Regina Rotman, \emph{Upper bounds on the length of the shortest closed geodesic on simply connected manifolds}. Math. Z. 233(2000), no.~2, 365--398.

\bibitem[SY17]{SY2017}
Richard Schoen, Shing-Tung Yau, \emph{Positive Scalar Curvature and Minimal Hypersurface Singularities}, 	arXiv:1704.05490.

\bibitem[Sim76]{Simon1976}
Leon Simon, \emph{Remarks on curvature estimates for minimal hypersurfaces}, Duke Math. J. 43(1976), no.~3, 545--553.

\bibitem[Sma93]{Smale93}
Nathan Smale, \emph{ Generic regularity of homologically area minimizing hypersurfaces in eight-dimensional manifolds}. Comm. Anal. Geom. 1(1993), no. 2, 217-228.

\bibitem[Zhu19]{Z2019}
Jintian Zhu, \emph{Rigidity of area-minimizing 2-sphere in $n$-manifolds with positive scalar curvature}. Proc. Amer. Math. Soc. 148(2020), no.8, 3479-3489.

\bibitem[Zhu20a]{Zhu2020a}
Jintian Zhu, \emph{Width estimate and doubly warped product}, Trans. Amer. Math. Soc. 374 (2021), no. 2, 1497–1511.

\bibitem[Zhu20b]{Zhu2020b}
Jintian Zhu, \emph{Rigidity results for complete manifolds with nonnegative scalar curvature}, arXiv:2008.07028.


\end{thebibliography}
\end{document}